\documentclass[a4paper,12pt,centertags]{amsart}

\usepackage[utf8]{inputenc}
\usepackage[sc,osf]{mat hpazo}
\usepackage[euler-digits]{eulervm}
\usepackage[T1]{fontenc}
\usepackage{mathrsfs}
\usepackage{amssymb}
\usepackage{amsmath}
 \usepackage[foot]{amsaddr}
\usepackage{mathtools}
\usepackage{xspace}
\usepackage{thmtools}
\usepackage[a4paper,centering]{geometry}
\usepackage[pdfdisplaydoctitle,colorlinks,breaklinks,urlcolor=blue,linkcolor=blue,citecolor=blue]{hyperref}
\usepackage{enumitem}
\usepackage[style=alphabetic]{biblatex}
\usepackage{xcolor}
\usepackage{tcolorbox}
\usepackage[nameinlink,capitalise]{cleveref}

\newcommand{\N}{\mathbb{N}}
\newcommand{\R}{\mathbb{R}}
\newcommand{\T}{\mathbb{T}}

\newcommand{\e}{\operatorname{e}}		
\newcommand{\uno}{\mathbb{1}}			

\newcommand{\Z}{\mathbb{Z}}
\DeclareMathOperator{\pt}{\partial_t}	
\DeclareMathOperator{\grad}{\nabla}		
\DeclareMathOperator{\sg}{\cdot \nabla}		

\newcommand{\norm}[2]{\|#1\|_{#2}}

\newcommand{\co}{\omega} 
\newcommand{\coLo}{\omega^{loc}} 
\newcommand{\coL}{\overline{\omega}} 
\newcommand{\triple}{\Delta^2_T} 
\newcommand{\locspac}{C^{p\var}_{\coL,L}} 
\newcommand{\locspace}{C^{p\var,2}_{\coL,L}} 
\newcommand{\parloc}{\boldsymbol{\Pi}(\coL,L)}
\newcommand{\rp}{\mathbf{Z}}
\newcommand{\var}{\textup{-var}}
\newcommand{\rpspace}{\mathscr{C}^{p\var}_g}
\newcommand{\contr}{\mathscr{D}^{p\var}_{Z,(\coL,L)}}

\newcommand{\ovsi}{\overline{\sigma}} 
\newcommand{\st}{_{s,t}}

\newtheorem{theorem}{Theorem}[section]
\newtheorem{lemma}[theorem]{Lemma}
\newtheorem{proposition}[theorem]{Proposition}
\newtheorem{corollary}[theorem]{Corollary}
\theoremstyle{definition}
\newtheorem{definition}[theorem]{Definition}

\theoremstyle{remark}
\newtheorem{remark}[theorem]{Remark}
\numberwithin{equation}{section}

\hypersetup{%
	pdftitle={Well posedness of 2D Euler equation driven by rough path noise},
	pdfauthor={L. Roveri, F. Triggiano},
	pdfcreator={L. Roveri}
}

\begin{document}
	\title{Well-posedness of rough 2D Euler equation with bounded vorticity}
	
    \author{L. Roveri$^1$}
    \address{$^1$ Dipartimento di Matematica, Universit\`a di Pisa, Largo Bruno Pontecorvo 5, I--56127 Pisa, Italia }\email{\href{mailto:leonardo.roveri@phd.unipi.it}{leonardo.roveri@phd.unipi.it}}
	
    \author{F. Triggiano$^2$} 
    \address{$^2$ Scuola Normale Superiore, Piazza dei Cavalieri, 7, I--56126 Pisa, Italia}\email{\href{mailto:francesco.triggiano@sns.it}{francesco.triggiano@sns.it}}
    \date{\today}
    \begin{abstract}
        We consider the 2D Euler equation with bounded initial vorticity and perturbed by rough transport noise. We show that there exists a unique solution, which coincides with the starting condition advected by the Lagrangian flow. Moreover, the stability of the solution map with respect to the initial vorticity and the rough perturbation yields a Wong-Zakai result for fractional Brownian driving paths.
    \end{abstract}
	\maketitle

	
\section{Introduction}
In the last two decades, there has been a significant development in the analysis of fluid dynamics equations perturbed by a transport type noise due to the emergence of the regularization by noise phenomena, namely the ability to restore well-posedness \cite{FGP10,CM23,GL23,FL21} or the appearance of a spare dissipative effect under a suitable scaling limit \cite{FGL21,BL24}.

While all the above mentioned works are restricted to the case of Brownian motion, similar results have been identified even with less regular noise such as fractional Brownian motion with Hurst parameter $H<\frac{1}{2}$ \cite{CG16,Cat16,GHM23}. Since the latter is not a semimartingale, classical stochastic calculus results are unavailable, and the time integral must be built in the rough integral sense. This suggests that rough path theory \cite{Lyo98,FH20,FV10} may offer the right tools to treat fluid dynamics equations perturbed by transport type noise with driving signals very irregular in time \cite{Hai11,GT10}.

Lately, the introduction of the "unbounded rough drivers" formalism (\cite{BG17, DEYA20193577}) has allowed the recovery of classical tools, such as the Gr\"onwall lemma, in the rough path setting. This approach has been successfully applied to obtain well-posedness results and energy estimates for Navier-Stokes and Euler equations perturbed by more general noises \cite{Hof1,Hof2,Cri1}.

In this work, we exploit the unbounded rough drivers' machinery to study the 2D Euler equation in vorticity form, perturbed by rough transport noise, in the Yudovich setting \cite{Y63}, namely with bounded starting condition:
\begin{equation}\label{eq:origEuler}
\begin{cases}
\pt w(t,x) + u(t,x) \sg w(t,x) = \sum_{j=1}^M \sigma_j(x) \sg w(t,x) \dot{Z}^j_t , \\
w(0,x) = w_0(x) , \quad u(t, \cdot) = K_{BS} * w(t,\cdot) ,
\end{cases}
\end{equation}
where $x \in \T^2$, $w_0 \in L^{\infty}(\T^2)$, $\sigma_j:\T^2\rightarrow\R^2$ are sufficiently smooth and divergence-free vector fields, $Z^j$ are finite $p$-variation paths with $p\in[2,3)$, and $K_{BS}$ denotes the Biot-Savart kernel.
In particular, we show that there exists a unique weak solution to \cref{eq:origEuler} and for the corresponding Lagrangian equation. Moreover, we also obtain that the solution map is continuous with respect to the initial condition, the $\ sigma_j$ and the driver.
 
From a technical point of view, the main novelties regard the uniqueness proof for both the Lagrangian RDE with log-Lipschitz drift and the 2D Euler equation. 
For the former, we prove the solution (in the sense of controlled rough paths, see \cite{Gub04}) is unique by applying a localization procedure together with a rough Gr\"onwall Lemma. For the latter, we obtain uniqueness by showing that any solution is a Lagrangian solution, namely it is advected by the associated flow. This claim is proved by combining Dobrushin's strategy from \cite{D79}, i.e. identifying the rough equation satisfied by any observable of the inverse flow map, with a doubling variable scheme. 

We conclude the introduction with the notion of solution and the main result. The rest of the paper is organized as follows: \cref{sec:preliminaries} contains the basics of the unbounded rough drivers and some results on $p$-variation spaces; in \cref{sec:Exist} we prove that there exists a solution to 2D Euler equation; \cref{sec:ELflow} is about the well-posedness results for the Lagrangian RDE and, at last, in \cref{sec:uniq} we conclude the uniqueness for the 2D Euler equation and report a Wong-Zakai result.

\subsection{Main results}\label{sec:formEq}

Since the driving signal in \cref{eq:origEuler} is too irregular to be integrated à la Young, the strategy is to iterate the equation, substituting the solution term $w$ in $\sigma_j \sg w \dot{Z}^j$. 
This way, we obtain  
\begin{align*}
w_{s,t} 
={} & 
- \int_s^t u_r \sg w_r d r 
+ \int_s^t \sigma_j \sg w_r d Z^j_r \\
={} & 
- \int_s^t u_r \sg w_r d r 
+ \sigma_j \sg w_s Z^j_{s,t} 
- \int_s^t (\sigma_j \sg) \int_s^r u_{r_1} \sg w_{r_1} d r_1 d Z^j_r \\
& + \int_s^t (\sigma_j \sg) \int_s^r \sigma_i \sg w_{r_1} d Z^i_{r_1} d Z^j_r \\
={} & 
- \int_s^t u_r \sg w_r d r 
+ \sigma_j \sg w_s Z^j_{s,t} 
- \int_s^t (\sigma_j \sg) \int_s^r u_{r_1} \sg w_{r_1} d r_1 d Z^j_r \\
& + (\sigma_j \sg) (\sigma_i \sg) w_s \Z^{i,j}_{s,t} 
- \int_s^t (\sigma_j \sg) \int_s^r (\sigma_i \sg) \int_s^{r_1} u_{r_2} \sg w_{r_2} d r_2 d Z^i_{r_1} d Z^j_r \\ 
& + \int_s^t (\sigma_j \sg) \int_s^r (\sigma_i \sg) \int_s^{r_1} \sigma_h \sg w_{r_2} d Z^h_{r_2} d Z^i_{r_1} d Z^j_r .
\end{align*}

The above can be reformulated as 
\begin{equation}\label{eq:rougheuler}
w_{s,t} = \mu_{s,t} + A^1_{s,t} w_s + A^2_{s,t} w_s + w^\natural_{s,t} ,
\end{equation}
where, for the sake of notation, we denoted the drift by 
\[
\mu_{s,t} = - \int_s^t u_r \sg w_r d r ,
\]
the unbounded rough operators by 
\begin{equation}\label{eq:urd1}
\begin{gathered}
A^1_{s,t} \varphi = \sigma_j \sg \varphi Z^j_{s,t} \\
A^2_{s,t} \varphi = (\sigma_j \sg) (\sigma_i \sg) \varphi \Z^{i,j}_{s,t}
\end{gathered}
\end{equation}
and the remainder term by 
\begin{equation}\label{eq:remainder}
\begin{aligned}
w^\natural_{s,t}
={} & 
- \int_s^t \int_s^r (\sigma_j \sg) (u_{r_1} \sg) w_{r_1} d r_1 d Z^j_r \\
& - \int_s^t \int_s^r \int_s^{r_1} (\sigma_j \sg) (\sigma_i \sg) (u_{r_2} \sg) w_{r_2} d r_2 d Z^i_{r_1} d Z^j_r \\ 
& + \int_s^t \int_s^r \int_s^{r_1} (\sigma_j \sg) (\sigma_i \sg) (\sigma_h \sg) w_{r_2} d Z^h_{r_2} d Z^i_{r_1} d Z^j_r .
\end{aligned}\end{equation}

\begin{remark}
	It is important to notice that the remainder term $w^\natural$ is defined starting from a solution $w$, and what one has to verify is that $w^\natural_{s,t}$ is in fact a negligible term, i.e. it has finite $\frac{p}{3}$-variation.
\end{remark}

We are ready to state what we mean by a solution of \cref{eq:rougheuler}.
Since we work with duals of Sobolev spaces $W^{n,q}$, we denote them by $W^{-n,q}$. In addition, we will denote $L^{\infty}([0,T]\times \T^2)$ and $C([0,T],X)$ by $L^{\infty}_{t,x}$ and $C_tX$, respectively
\begin{definition}\label{def:roughsolution}
	A function $w:[0,T]\times \T^2 \to \R$ is a weak solution of \cref{eq:rougheuler} if $w\in C_tW^{-1,1} \cap L^\infty_{t,x}$ and there exists a localization $(\coL,L)$ such that $w^\natural \in C_{2,\coL,L}^{p/3\var}(W^{-3,1})$, i.e.
	\[
	w^\natural_{s,t}(\psi) \coloneqq w_{s,t}(\psi) - \mu\st (\psi) - w_s \left( \left(  A^{1,*}_{s,t} + A^{2,*}_{s,t} \right) \psi \right)
	\]
	has finite $\frac{p}{3}$-variation with respect to $(\coL, L)$ for all $\psi\in W^{3,1}$.
\end{definition}

We can, then, state the main results.
\begin{theorem}\label{mainthm}
    Let $w_0\in L^{\infty}(\T^2)$, $\mathbf{Z}$ be a $p$-geometric rough path with $p \in [2,3)$ and $\{\sigma_j\}_{j=1}^M\subset C^3(\T^2,\R^2)$ be divergence-free vector fields. Then, the following statements hold:
    \begin{enumerate}
        \item (\cref{ExistenceLagr}) If $w$ is a solution in the sense of \cref{def:roughsolution} and $u=K_{BS}\ast w$, then there exists a unique controlled rough path $\phi_t$ w.r.t. $\mathbf{Z}$ such that
        \begin{equation*}
            \phi_t=Id_{\T^2}+\int_0^t u(r,\phi_r)dr-\int_0^t \sigma_j(\phi_r)dZ_r^j;
        \end{equation*}
        \item (\cref{th:vanishvisc,thm:uniq,thm:WZ}) There exists a unique solution $w$ to \cref{eq:origEuler} in the sense of \cref{def:roughsolution}. Moreover, $w$ is advected by the associated flow,
        \begin{equation*}
            w_t=(\phi_t)_{\#}w_0,
        \end{equation*}
        and the solution map $(w_0,\{\sigma_j\}_j,\mathbf{Z})\mapsto w$ is continuous.
    \end{enumerate}
\end{theorem}

\section{Preliminaries}\label{sec:preliminaries}


\subsection{Localized variation spaces}

Since it emerges naturally in rough path theory, let us recall what we mean by the term ``control''. 
For a deeper description of its properties, we refer e.g. to \cite{FV10}. 

From now on, we are going to denote 
\[
\Delta_T \coloneqq \{(s,t) | 0\leq s \leq t \leq T\} , 
\qquad 
\Delta_T^2 \coloneqq \{(s,r,t) | 0\leq s \leq r \leq t \leq T\} ,
\]
and, for two-index functions $g:\Delta_T\to V$, $(s,r,t)\in\Delta_T^2$,
\[
\delta g_{s,r,t}  \coloneqq g(s,t) - g(s,r) - g(r,t) .
\]

In addition, throughout our discussion, we are going to use the notation $a \lesssim b$ whenever there is an inequality up to a (positive) constant, i.e. there exists some $C>0$ such that $a \leq C b$. When the constant depends on a parameter, $C = C(par)$, we highlight that as follows: $a \lesssim_{par} b$.

\begin{definition}
    A continuous function $\co:\Delta_T \to \R^+$ is a control if 
    \begin{itemize}
        \item $\co(s,s)=0$, for all $s\in[0,T]$;
        \item $\co$ is super-additive, i.e. $\co(s,u)+\co(u,t)\le \co(s,t)$ for all $(s,u,t)\in \triple$.
    \end{itemize}
    
    Moreover, a control $\omega$ is said to be strictly positive if, for all $(s,t) \in \Delta_T$, 
    \begin{equation*}
        \omega(s,t)=0 \implies s=t.
    \end{equation*}
\end{definition}

Another key concept is the $p$-variation of a path, which can be seen as a generalization of H\"older continuity.

\begin{definition}
    Given a Banach space $V$, $p>0$ and a continuous function $g:[0,T]\to V$, we say that $g$ has finite $p$-variation, $g\in C^{p\var}([0,T],V)$, if 
    \[
    \norm{g}{p}^p \coloneqq \sup_{\Pi\in\boldsymbol{\Pi}([0,T])} \sum_{[s,t]\in\Pi} |g_t-g_s|^p < \infty ,
    \]
    where $\boldsymbol{\Pi}([0,T])$ is the set of partitions of $[0,T]$, $g_t\coloneqq g(t)$ and $|\cdot| = \norm{\cdot}{V}$ (we are going to adopt this notation for a norm whenever the space $V$ is clear from the context).

    In the case of a two-index function $g:\Delta_T\to V$, the above notation can be adapted by changing $|g_t-g_s|$ into $|g_{s,t}|$, where $g\st\coloneqq g(s,t)$.
\end{definition}

In particular, we are going to extensively use partitions whose mesh is bounded from above by some quantity. 
To this aim, we introduce the concept of \textit{localized $p$-variation}.

\begin{definition}
    Given a Banach space $V$, $p>0$, a control $\coL$ and $L>0$, we denote by $\locspace(V)$ the space of $2$-index continuous functions $g:\Delta_T\to V$
    such that 
    \begin{equation} \label{pNorm}
        \norm{g}{p,(\coL,L),[0,T]}^p\coloneqq \sup_{ \Pi \in \parloc} \sum_{[s,t] \in \Pi} |g_{s,t}|^p<\infty,
    \end{equation}
    where $\parloc$ denotes the space of partitions $\Pi$ of $[0,T]$ such that $\coL(s,t)\le L$ for all $[s,t]\in \Pi$.

    Analogously, $\locspac(V)$ denotes the space of $1$-index continuous functions $f:[0,T]\to V$ such that \cref{pNorm} holds, with $g_{s,t}\coloneqq |g_{t}-g_{s}|$.
\end{definition}

By slightly modifying some results in \cite{FV10} (see, e.g., Proposition 5.8 and Proposition 5.10), we obtain the following Proposition.

\begin{proposition}\label{BestControl}
Let $g \in \locspace(V)$ or $g \in \locspac(V)$. Then, the function
$\co_g:\Delta_T\to \R^+$ defined by
\begin{equation*}
    \coLo _g(s,t)=\norm{g}{p,(\coL,L),[s,t]}^p
\end{equation*}
is a control. In addition, $\coLo_g$ is the smallest control such that 
\begin{equation}\label{SmallestControlCondition}
    |g_{s,t}|\le \coLo_g(s,t)^{\frac{1}{p}} \text{ for all } (s,t)\in \Delta_T \text{ such that } \coL(s,t)\le L.
\end{equation}
Namely, if $\co$ is a control which satisfies \cref{SmallestControlCondition}, then \begin{equation*}
    \coLo_g(u,v)\le \co(u,v),
\end{equation*} for all $(u,v)\in \Delta_T$.
\end{proposition}

We conclude this introduction to localized $p$-variation spaces by showing a result that holds only for $1$-index functions. Let us recall that $p\geq 1$ is the only interesting case for $1$-index functions (see \cite[Proposition 5.2]{FV10}).
\begin{proposition}\label{EstimateControl}
    Let $g\in \locspac(V)$. Then, the following hold:
    \begin{enumerate}
        \item $\locspac=C^{p\var}$, i.e. if $g$ has finite localized $p$-variation with respect to $(\coL,L)$, then it has finite $p$-variation. Moreover, the following inequalities hold
        \begin{align}&\label{ControlloPerStimaApriori} \norm{g}{p,(\coL,L),[s,t]}^p\le 2^{p-1} \norm{g}{p,(\coL,\frac{L}{2}),[s,t]}^p \quad \forall (s,t)\in \Delta_T;\\
        &\label{CompareVar} 
        \norm{g}{p,[0,T]}^p\lesssim \Big(\frac{L}{\coL(0,T)}\Big)^{p-1}\norm{g}{p,[0,T],(\coL,L)}^p.   
        \end{align}

        \item If $\coL$ is a strictly positive control, then
        \begin{equation*}
            \lim_{L\to 0}\norm{g}{q,(\coL,L),[s,t]}=0,
        \end{equation*}
        for all $q>p\geq 1$, $(s,t)\in \Delta_T$.
    \end{enumerate}
\end{proposition}
\begin{proof}
\begin{enumerate}
    \item Let us prove \cref{ControlloPerStimaApriori}. The second inequality, \cref{CompareVar}, can be easily obtained by properly iterating the same strategy.

    Given $(s,t) \in \Delta_T$, consider a partition coherent with the localization, namely $P=\{s\eqqcolon t_0,t_1,\dots,t_N\coloneqq t\}\in \parloc$.\\
    Then, the continuity of $\coL$ implies that we can find $s_i\in (t_i,t_{i+1})$ such that $\coL(t_i,s_i)\le \frac{L}{2}$ and $\coL(s_i,t_{i+1})\le \frac{L}{2}$, for all $i=0,\dots,N-1$. Therefore, $Q=P\cup\{s_i\}_i$ is a partition of $[s,t]$ that satisfies the localization $(\coL,\frac{L}{2})$.
    At last, the additivity of $1$-index functions' variation, i.e. $g_{s,t}=g_{s,u}+g_{u,t}$, implies that
    \begin{equation*}
        \sum_{t_i\in P}|g_{t_i,t_{i+1}}|^p\le 2^{p-1}\sum_{r_i \in Q}|g_{r_i,r_{i+1}}|^p.
    \end{equation*}
    The obtained inequality proves the statement.

    \item Notice that
    \begin{equation*}
        \sum_{[s,t]\in P}|g\st|^q\le \norm{g}{p}^p \sup_{(s,t)\in \Delta_T, \coL(s,t)\le L}|g\st|^{q-p},
    \end{equation*}
    for any partition $\Pi\in \parloc$.
    
    Now, let us prove by contradiction that $\lim_{L\to 0}h(L)=0$, where 
    \begin{equation*}
        h(L)=\sup_{(s,t)\in \Delta_T,\coL(s,t)\le L}|g\st|.
    \end{equation*}
    Suppose that $\lim_{L\to 0}h(L)=\delta >0$. Then for all $n \in N$, there exists $(s_n,t_n)\in \Delta_T$ such that $\coL(s_n,t_n)\le \frac{1}{n}$ and $|g_{s_n,t_n}|\geq \frac{\delta}{2}$. Therefore, up to taking a subsequence, there exists $(s,t) \in \Delta_T$ such that 
    \begin{equation*}
        (s_n,t_n)\to (s,t),\quad \coL(s_n,t_n)\to \coL(s,t), \quad |g\st|\geq \frac{\delta}{3.12}.
    \end{equation*}
    So, we find a contradiction because $\coL(s,t)=0$ and $\coL$ is strictly positive.
\end{enumerate}
\end{proof}
\begin{remark}
    The assumptions $q>p$ and $\coL$ being strictly positive cannot be weakened. Indeed, $g(t)=\sqrt{t}$, for $t\in [0,1]$, shows that the third claim is not true for $q=p$. Instead, $\coL\equiv 0$ indicates that the same claim does not hold without requiring strict positiveness.
\end{remark}


\subsection{Localized controlled rough paths}
In this section, we recall the definition of rough paths. 
We also introduce a generalized version of controlled rough paths as defined in \cite{Gub04}.

\begin{definition}
Let $M\in \N$, $p\in[2,3)$. A $p$-rough path is a pair
    \begin{center}
        $\mathbf{Z}=(Z,\Z)\in C^{p\var}(\R^M)\times C^{\frac{p}{2}\var,2}(\R^{M\times M}) $
    \end{center}  
    that satisfies the Chen's relation, namely
    \begin{equation*}
        \Z\st-\Z_{s,u}-\Z_{u,t}=Z_{s,u}\otimes Z_{u,t}, \quad \forall (s,u,t) \in \triple.
    \end{equation*}
Moreover, a $p$-rough path $\mathbf{Z}$ is said geometric if there exists a sequence $\{Z^n\}_n\subset C^{\infty}([0,T],\R^M)$ such that
$\mathbf{Z}^n\st\coloneqq (Z^n\st,\int_s^t Z^n_{s,r}\otimes dZ^n_r)$ converges to $\mathbf{Z}$ with respect to the product topology.
\end{definition}
\begin{remark}
    For any $p$-rough path there exists a control $\omega_Z$ such that
    \begin{center}
        $|Z\st|^p\le \omega_Z(s,t),\quad |\Z\st|^{\frac{p}{2}}\le \omega_Z(s,t)$.
    \end{center}
\end{remark}

\begin{definition}Consider a $p$-rough path $\mathbf{Z}$, a control $\coL$, a positive constant $L$, a Banach space $V$ and denote the space of continuous and linear operators from $\R^M$ to $V$ by $L(\R^M,V)$.

A continuous function $X\in C^{p\var}([0,T],V)$ is a localized controlled rough path with respect to $\mathbf{Z}$ and $(\coL,L)$ if there exists a pair $(X',R^X)\in C^{p\var}(L(\R^M,V))\\\times C_{\coL,L}^{\frac{p}{2}\var,2}(V)$ such that
\begin{equation}
    X\st=X'_sZ\st+R^X\st,\quad \forall (s,t)\in\Delta_T.
\end{equation}
The space of localized controlled rough paths will be denoted by $\contr(V)$.
\end{definition}

The rough integral can be built through \cref{SewingLemma}. 
We report some useful inequalities that can be proved by easily generalizing the corresponding results in \cite{FH20}.

\begin{proposition}\label{ContRouInt}
    Consider two $p$-rough paths $\rp^1,\rp^2$, and two controlled rough paths $(X,X')\in \mathscr{D}^{p\var}_{\rp^1,(\overline{\omega}_1,L_1)}(L(\R^M,V))$, $(Y,Y')\in \mathscr{D}^{p\var}_{\rp^2,(\overline{\omega}_2,L_2)}(L(\R^M,V))$. Then,
    \begin{enumerate}
        \item $(\int_0^tX_rdZ^1_r,X_t)\in \mathscr{D}^{p\var}_{\rp^1,(\overline{\omega}_1,L_1)}$. Moreover, the following estimate holds
        \begin{multline}\label{RouIntEstimate}
            \Big|\int_s^tX_rd\mathbf{Z}_r-X_sZ\st-X'_s\Z\st\Big| \\
            \lesssim 
            \co_{R^X}^{\text{loc}}(s,t)^{\frac{2}{p}}\co_Z(s,t)^{\frac{1}{p}}+\co_{X'}^{\text{loc}}(s,t)^{\frac{1}{p}}\co_Z(s,t)^{\frac{2}{p}}, 
        \end{multline}
        for all  $(s,t)\in \Delta_T$ such that $\overline{\omega}_1(s,t)\le L_1$.

        \item The following inequality holds
        \begin{equation*}
        \begin{aligned}
            \Big\|\int_0^{\cdot}X_rdZ^1_r-\int_0^{\cdot}Y_rdZ^2_r\Big\|_{p,[0,T]}^p&\lesssim_{L_1,L_2} |X-Y|^p_{C_t}\co_{Z^1}(0,T)+
            |X|^p_{C_t}\co_{Z^1-Z^2}(0,T)\\
            &+|X'|_{C_t}^p\co_{Z^1-Z^2}(0,T)^2+|X'-Y'|_{C_t}^p\co_{Z^2}(0,T)^2\\
            &+\co_{Z^1}(0,T)\co^{\text{loc}}_{R^X-R^Y}(0,T)^2+\co_{Z^1-Z^2}(0,T)\co^{\text{loc}}_{R^Y}(0,T)^2\\
            &+\co_{Y'}^{\text{loc}}(0,T)\co_{Z^1-Z^2}(0,T)^2+\co_{Z^2}(0,T)^2\co^{\text{loc}}_{X'-Y'}(0,T),
        \end{aligned}
        \end{equation*}
        where $|f|_{C_t}$ denotes $\sup_{s \in [0,T]}|f(s)|$.
    \end{enumerate}
\end{proposition}


\subsection{Unbounded rough drivers and 
smoothing operators}\label{sec:URD&SO}
In order to rewrite our equation \eqref{eq:origEuler} in a more treatable form, we rewrite the integral against $Z$ as a Taylor expansion up to order two. In particular, this decomposition produces some sort of ``operator-valued'' rough path, i.e. linear unbounded operators enjoying rough path properties, and an additional remainder.

Let $\{E^n\}_{n\in\N}$ be a family of Banach spaces such that there is a continuous embedding $E^m \subseteq E^n$ for all $m \geq n$. 
From now on, we denote their dual spaces by $E^{-n}$. 

The ``operator-valued'' rough path can be formally defined as follows.
\begin{definition}
	Let $p \in [2,3)$. A continuous unbounded $p$-rough driver on the scales $\{E^n\}_{n\in\N}$ is a pair $(A^1,A^2)$ of $2$-index maps such that, for all $s,t\in\Delta_T$, 
	\begin{gather*}
	A^1_{s,t} \in L(E^{-n}, E^{-n-1}) \text{ for } n \in \{0,1,2\} , \\
	A^2_{s,t} \in L(E^{-n}, E^{-n-2}) \text{ for } n \in \{0,1\} .
	\end{gather*}
	Moreover, Chen's relation holds
	\begin{equation}\label{eq:chenurd}
	A^1_{s,t} = A^1_{s,r} + A^1_{r,t} \ 
	\quad 
	A^2_{s,t} = A^2_{s,r} + A^2_{r,t} + A^1_{r,t} A^1_{s,r}, 
	\qquad 
	\forall (s,r,t)\in\triple 
	\end{equation}
	 and there exists a control $\omega_A$ such that 
	\begin{equation}\label{eq:controlurd}
	\norm{A^1_{s,t}}{L(E^{-n}, E^{-n-1})} \leq \omega_A(s,t)^{\frac{1}{p}} , 
	\qquad 
	\norm{A^2_{s,t}}{L(E^{-n}, E^{-n-2})} \leq \omega_A(s,t)^{\frac{2}{p}} .
	\end{equation}
\end{definition}

We also need to introduce the concept of smoothing operators. They are crucial for estimating the time regularity of the remainder.

\begin{definition}\label{def:smothing}
    A family of operators $(J^{\eta})_{\eta \in [0,1]}$ is said to be a family of smoothing operators on the scale $\{E^n\}_{n\in \N}$ if the following inequalities hold
    \begin{align}
        &\norm{J^{\eta}\psi}{E^{n+m}}\lesssim\eta^{-m} \norm{\psi}{E^{n}},\\
        &\norm{(I-J^{\eta})\psi}{E^{n}}\lesssim\eta^{m} \norm{\psi}{E^{n+m}},
    \end{align}
     for all $\psi \in E^{n+m}$ and $n,m\in \N$.
\end{definition}


\section{Existence of solutions}\label{sec:Exist}
In this section, we aim to show the existence of solutions in the sense of \cref{def:roughsolution}.

\subsection{A priori estimates}

We start with a priori estimates on solutions and their remainders. To prove them, we use the following smoothing operators. 

\begin{lemma}\label{smoothing1}
    Consider a (symmetric) mollifier $\rho$ with support on the ball of radius equal to 1. Then, $(J^{\eta})_{\eta \in [0,1]}$, defined by $J^{\eta}\psi=\rho_{\eta}\ast \psi$, is a family of smoothing operators on the scale $W^{n,1}(\T^2)$ in the sense of \cref{def:smothing}.
\end{lemma}

The first estimate is on the remainder $w^{\natural}$, defined in \cref{def:roughsolution}. We denote by $\omega_{\natural}\coloneqq \omega_{w^{\natural}}^{\text{loc}}$.
\begin{lemma}\label{lemma:apriori}
	Let $w$ be a solution of the problem, in the sense of \cref{def:roughsolution}. Then there exists $\overline{L}>0$ such that
	\begin{equation}\label{eq:aprioriremainder}
	\omega_\natural(s,t) 
	\lesssim_p \norm{w}{L^\infty_{t,x}}^{\frac{p}{3}} \omega_A(s,t)
	+ \norm{w}{L^\infty_{t,x}}^{\frac{2p}{3}}  |t-s|^{\frac{p}{3}} \left( \omega_A(s,t)^{\frac{1}{3}} + \omega_A(s,t)^{\frac{2}{3}} \right),
	\end{equation}
	for all $(s,t) \in \Delta_T$ s.t. $\coL(s,t)\le L$ and $\co_{A}(s,t)\le \overline{L}$.
\end{lemma}

\begin{proof}
	Let $(s,r,t) \in \Delta_T^2$ such that $\coL(s,t)\le L$. It follows from \cref{eq:rougheuler} that, for all $\psi\in W^{3,1}$, 
	\begin{align*}
	\delta w ^\natural_{s,r,t}(\psi)
	={} & - w_s \left( \left[ A^{1,*}_{s,t} + A^{2,*}_{s,t} \right] \psi \right) + w_s \left( \left[ A^{1,*}_{s,r} + A^{2,*}_{s,r} \right] \psi \right) \\ 
    & + w_r \left( \left[ A^{1,*}_{r,t} + A^{2,*}_{r,t} \right] \psi \right) \\
	={} & w_{s,r} \left( A^{2,*}_{r,t} \psi \right) + \left( w_{s,r} - w_s A^{1,*}_{s,r} \right) \left( A^{1,*}_{r,t} \psi \right) \\
	={} & w_{s,r} \left( A^{2,*}_{r,t} \psi \right) + w^\dagger_{s,r} \left( A^{1,*}_{r,t} \psi \right) ,
	\end{align*}
	where in particular we have exploited \cref{eq:chenurd} and defined 
	\begin{equation}\label{eq:wdagger}
	w^\dagger_{s,t} \coloneqq w_{s,t} - A^1_{s,t} w_s = \mu_{s,t} + A^2_{s,t} w_s + w^\natural_{s,t} .
	\end{equation}
	
	Let now $J^\eta$ be a smoothing operator, as defined in \cref{smoothing1}: we are going to apply it to $\psi$ in order to separately analyze a rougher and a smoother part of its duality product against $\delta w^\natural$. Starting with the former,
	\begin{align*}
	| \delta w^\natural_{s,r,t} \left( (I-J^\eta) \psi \right) | 
	\leq{} & 
	\left| w_{s,r} \left( A^{2,*}_{r,t} (I-J^\eta) \psi \right) \right| + \left| w^\dagger_{s,r} \left( A^{1,*}_{r,t} (I-J^\eta) \psi \right) \right| \\
	\lesssim{} & 
	\norm{w}{L^\infty_{t,x}} \Big[ \norm{A^{2,*}_{r,t} (I-J^\eta) \psi}{L^1} + \norm{A^{1,*}_{r,t} (I-J^\eta) \psi}{L^1} \\
	& + \norm{A^{1,*}_{s,r} A^{1,*}_{r,t} (I-J^\eta) \psi}{L^1} \Big] \\
	\lesssim{} & 
	\norm{w}{L^\infty_{t,x}} \Big[ \omega_A(s,t)^{\frac{2}{p}} \norm{(I-J^\eta) \psi}{W^{2,1}} \\
    & + \omega_A(s,t)^{\frac{1}{p}} \norm{(I-J^\eta) \psi}{W^{1,1}} \Big] \\
	\leq & 
	\norm{w}{L^\infty_{t,x}} \Big[ \omega_A(s,t)^{\frac{2}{p}} \eta + \omega_A(s,t)^{\frac{1}{p}} \eta^2 \Big] \norm{\psi}{W^{3,1}} .
	\end{align*}
	
	As for the smoothed part,
	\begin{align*}
	| \delta w^\natural_{s,r,t} \left( J^\eta \psi \right)
	| 
	\leq{} & 
	\left| w_{s,r} \left( A^{2,*}_{r,t} J^\eta \psi \right) \right| + \left| w^\dagger_{s,r} \left( A^{1,*}_{r,t} J^\eta \psi \right) \right| \\
	\leq{} & 
	\left| \mu_{s,r} \left( A^{2,*}_{r,t} J^\eta \psi \right) \right| + \left| w_s \left( A^{1,*}_{s,r} A^{2,*}_{r,t} J^\eta \psi \right) \right| + \left| w_s \left( A^{2,*}_{s,r} A^{2,*}_{r,t} J^\eta \psi \right) \right| \\
	& + \left| w^\natural_{s,r} \left( A^{2,*}_{r,t} J^\eta \psi \right) \right| 
	+ \left| \mu_{s,r} \left( A^{1,*}_{r,t} J^\eta \psi \right) \right| + \left| w_s \left( A^{2,*}_{s,r} A^{1,*}_{r,t} J^\eta \psi \right) \right| \\
	& + \left| w^\natural_{s,r} \left( A^{1,*}_{r,t} J^\eta \psi \right) \right| \\ 
	\leq{} & 
	\norm{w}{L^\infty_{t,x}}^2 \norm{A^{2,*}_{r,t} J^\eta \psi}{W^{1,1}}|t-s| + \norm{w}{L^\infty_{t,x}} \norm{A^{1,*}_{s,r} A^{2,*}_{r,t} J^\eta \psi}{L^1} \\
	& + \norm{w}{L^\infty_{t,x}} \norm{A^{2,*}_{s,r} A^{2,*}_{r,t} J^\eta \psi}{L^1} + \norm{w^\natural_{s,r}}{W^{-3,1}} \norm{A^{2,*}_{r,t} J^\eta \psi}{W^{3,1}} \\
	& + \norm{w}{L^\infty_{t,x}}^2 \norm{A^{1,*}_{r,t} J^\eta \psi}{W^{1,1}}|t-s| + \norm{w}{L^\infty_{t,x}} \norm{A^{2,*}_{s,r} A^{1,*}_{r,t} J^\eta \psi}{L^1} \\
	& + \norm{w^\natural_{s,r}}{W^{-3,1}} \norm{A^{1,*}_{r,t} J^\eta \psi}{W^{3,1}} .
	\end{align*}
	Due to the regularity properties of both the unbounded rough drivers and the smoothing operators, we can rewrite the estimate as 
	\begin{align*}
	| \delta w^\natural_{s,r,t} ( J^\eta \psi ) | 
	\leq{} & 
	\norm{w}{L^\infty_{t,x}}^2 \norm{J^\eta \psi}{W^{3,1}} \omega_A(s,t)^{\frac{2}{p}} |t-s| + \norm{w}{L^\infty_{t,x}} \norm{J^\eta \psi}{W^{3,1}} \omega_A(s,t)^{\frac{3}{p}} \\
	& + \norm{w}{L^\infty_{t,x}} \norm{J^\eta \psi}{W^{4,1}} \omega_A(s,t)^{\frac{4}{p}} + \omega_\natural(s,t)^{\frac{3}{p}} \norm{J^\eta \psi}{W^{5,1}} \omega_A(s,t)^{\frac{2}{p}} \\
	& + \norm{w}{L^\infty_{t,x}}^2 \norm{J^\eta \psi}{W^{2,1}} \omega_A(s,t)^{\frac{1}{p}} |t-s| + \norm{w}{L^\infty_{t,x}} \norm{J^\eta \psi}{W^{3,1}} \omega_A(s,t)^{\frac{3}{p}} \\
	& + \omega_\natural(s,t)^{\frac{3}{p}} \norm{J^\eta \psi}{W^{4,1}} \omega_A(s,t)^{\frac{1}{p}} \\
	\leq{} &
	\norm{w}{L^\infty_{t,x}}^2 \norm{\psi}{W^{3,1}} \omega_A(s,t)^{\frac{2}{p}} |t-s| + \norm{w}{L^\infty_{t,x}} \norm{\psi}{W^{3,1}} \omega_A(s,t)^{\frac{3}{p}} \\
	& + \eta^{-1} \norm{w}{L^\infty_{t,x}} \norm{\psi}{W^{3,1}} \omega_A(s,t)^{\frac{4}{p}} + \eta^{-2} \omega_\natural(s,t)^{\frac{3}{p}} \norm{\psi}{W^{3,1}} \omega_A(s,t)^{\frac{2}{p}} \\
	& + \norm{w}{L^\infty_{t,x}}^2 \norm{\psi}{W^{3,1}} \omega_A(s,t)^{\frac{1}{p}} |t-s| + \norm{w}{L^\infty_{t,x}} \norm{\psi}{W^{3,1}} \omega_A(s,t)^{\frac{3}{p}} \\
	& + \eta^{-1} \omega_\natural(s,t)^{\frac{3}{p}} \norm{\psi}{W^{3,1}} \omega_A(s,t)^{\frac{1}{p}} .
	\end{align*}
	
	Putting back together $(I-J^\eta)\psi$ and $J^\eta \psi$, we can set $\eta = \omega_A(s,t)^{1/p} \lambda$ for some $\lambda>0$ constant. This yields 
	\begin{multline*}
	\norm{\delta w^\natural_{s,r,t}}{W^{-3,1}}
	\lesssim{} 
	(\lambda^{-1} + 1 + \lambda + \lambda^2) \norm{w}{L^\infty_{t,x}} \omega_A(s,t)^{\frac{3}{p}} \\
	+ \norm{w}{L^\infty_{t,x}}^2 \left( \omega_A(s,t)^{\frac{1}{p}} + \omega_A(s,t)^{\frac{2}{p}} \right) |t-s| + (\lambda^{-2} + \lambda^{-1}) \omega_\natural(s,t)^{\frac{3}{p}} .
	\end{multline*}
	Now, \cref{ineq: corSewing} implies that 
	\begin{multline*}
	\omega_\natural(s,t) 
	\leq
	C_p 
	\Big[ (\lambda^{-1} + 1 + \lambda + \lambda^2)^{\frac{p}{3}} \norm{w}{L^\infty_{t,x}}^{\frac{p}{3}} \omega_A(s,t) \\
	+ \norm{w}{L^\infty_{t,x}}^{\frac{2p}{3}} \left( \omega_A(s,t)^{\frac{1}{3}} + \omega_A(s,t)^{\frac{2}{3}} \right) |t-s|^{\frac{p}{3}} + (\lambda^{-2} + \lambda^{-1})^{\frac{p}{3}} \omega_\natural(s,t) \Big] .
	\end{multline*}
	It is now sufficient to choose $\lambda$ in a way such that $C_p (\lambda^{-2} + \lambda^{-1})^{p/3} \leq \frac{1}{2}$ and $\bar{L}>0$ such that $\eta\leq1$ to obtain the thesis. 
\end{proof}

\begin{lemma}\label{lemma:apriorisol}
	Let $w$ be a solution of the problem, in the sense of \cref{def:roughsolution}. Then $w\in C^{p\var}_T W^{-1,1}$ and there exists $\overline{L}>0$ such that, given $\omega_w(s,t) \coloneqq \norm{w}{p\var,[s,t]}^p$,
	\begin{equation}\label{eq:apriorisol}
	\omega_w(s,t) 
	\lesssim_p \left( 1 + \norm{w}{L^\infty_{t,x}} \right)^{2p} \left( |t-s|^p + \omega_A(s,t) + \omega_\natural(s,t) \right) ,
	\end{equation}
	for all $(s,t) \in \Delta_T$ s.t. $\coL(s,t)\le L$ and $\co_{A}(s,t), \co_\natural(s,t)\le \overline{L}$.
\end{lemma}

\begin{proof}
Consider again a family of smoothing operators as in \cref{smoothing1}. Given any $\psi\in W^{1,\infty}$ analyzing separately $J^\eta\psi$ and $(I-J^\eta)\psi$ we get 
\[
|w_{s,t}((I-J^\eta)\psi)| 
\leq 
\norm{w_{s,t}}{L^\infty_{x}} \norm{(I-J^\eta)\psi}{L^1} 
\leq 
2 \eta \norm{w}{L^\infty_{t,x}} \norm{\psi}{W^{1,1}}
\] 
and, exploiting the formulation \eqref{eq:rougheuler},
\begin{align*}
|w_{s,t}(J^\eta\psi)| 
\leq{} &
\norm{w}{L^\infty_{t,x}}^2 \norm{J^\eta\psi}{W^{1,1}} + \norm{w}{L^\infty_{t,x}} \omega_A(s,t)^\frac{1}{p} \norm{J^\eta\psi}{W^{1,1}} \\
& + \norm{w}{L^\infty_{t,x}}\omega_A(s,t)^\frac{2}{p} \norm{J^\eta\psi}{W^{2,1}} + \omega_\natural(s,t)^\frac{3}{p} \norm{J^\eta\psi}{W^{3,1}} \\
\leq{} & 
\left( 1 + \norm{w}{L^\infty_{t,x}} \right)^2 \left( |t-s| + \omega_A(s,t)^{\frac{1}{p}} + \eta^{-1} \omega_A(s,t)^{\frac{2}{p}} + \eta^{-2} \omega_\natural(s,t)^{\frac{3}{p}} \right) .
\end{align*}
This, together with the choice $\eta = \omega_A(s,t)^{1/p} + \omega_\natural(s,t)^{1/p}$, yields the thesis whenever $\omega_A(s,t)$, $\omega_\natural(s,t)$ are less or equal of some $\bar{L}$ such that $\eta \leq 1$.
\end{proof}


\subsection{Existence}
We discuss here how to prove our existence result.
The strategy outline is classic: first, we prove the existence of solutions with an additional viscous dissipative term. 
Then, we prove the convergence of such solutions while letting the viscosity vanish, recovering a solution of the original problem, in the sense of Definition \ref{def:roughsolution}.

The proof of the following lemma can be easily obtained by adapting the results in \cite{BA94,Y63}.

\begin{lemma}
    Let $w_0\in C^{\infty}(\T^2)$, $\{\sigma_j\}_{j=1}^M\subset C^{\infty}(\T^2,\R^2)$ be a family of divergence-free vector fields and $Z\in C^{\infty}([0,T],\R^M)$. Then, there exists a unique solution $w \in C^{\infty}([0,T]\times \T^2)$ to the following equation
    
    \begin{equation}\label{eq:origNS}
    \begin{cases}
    \pt w + u \sg w -\nu \Delta w= \sum_{j=1}^M \sigma_j \sg w \dot{Z}^j , \\
    w(0) = w_0 , \quad
    u = K_{BS}*w.
     
    \end{cases}
    \end{equation}
    In particular, we have the estimate
    \begin{equation}\label{ineq:BoundInf}
        \norm{w}{L^{\infty}_{t,x}}\le C\norm{w_0}{L^{\infty}_x},
    \end{equation}
    where $C$ is a positive constant that does not depend on the viscosity $\nu$.
\end{lemma}

Thanks to the above, we can start with a smooth approximation of both solution and noise and let them converge.

\begin{theorem}\label{th:vanishvisc}
	Let $w_0 \in L^{\infty}(\T^2)$, $\{\sigma_j\}_{j=1}^M\subset C^3(\T^2,\R^2)$ be a family of divergence-free vector fields and $\mathbf{Z} \in \rpspace$. Then, there exists a solution to \cref{eq:origEuler} in the sense of \cref{def:roughsolution}.
\end{theorem}

\begin{proof}
    Since $\mathbf{Z}$ is a geometric rough path, there exists a sequence of smooth paths, $\{Z^N\}_{N\in\N}$, whose canonical lifts converge to $\mathbf{Z}$ in $\rpspace$.\\
    Let $w^N$ be the solution to \cref{eq:origNS} with initial datum  $w_0^N=\rho_{1/N}\ast w_0$, viscosity $\nu=N^{-1}$, $\sigma_j^N=\sigma_j\ast \rho_{1/N}$ and (smooth) driving signal $Z^N$.
    By iterating the equation, following the procedure already used in \cref{sec:formEq}, we can see that $w^N$ solves an equation of the form 
     \begin{equation}\label{eq:EulerApprox}
         w^N\st=\int_s^t \left[\frac{1}{N}\Delta w^N_r-u^N_r\sg w_r^N\right]dr+A^{N,1}\st w^N_s+A^{N,2}\st w^N_s+w^{N,\natural}\st,
     \end{equation}
     where 
     \[
     A^{N,1}\st w_s =\sigma^N_j \sg w_s Z^{N,j}\st , 
     \qquad 
     A^{N,1}\st w_s = (\sigma^N_j \sg) (\sigma^N_i \sg)w_s \Z^{N,i,j}\st
     \]
     and the remainder $w^{N,\natural}_{s,t}$ is properly defined via a suitable adaptation of \cref{eq:remainder}.
     
     Adapting the proofs of \cref{lemma:apriori} and \cref{lemma:apriorisol}, one can easily obtain that, uniformly in $N\in \N$,
     \begin{equation}\label{stimaRemUnif}
         \co_{N,\natural}(s,t)\lesssim_{p,w_0}\co_A(s,t)+|t-s|^{\frac{p}{3}}\big(\co_A(s,t)^{\frac{1}{3}}+\co_A(s,t)^{\frac{2}{3}}\big)
     \end{equation}
     and 
     \begin{equation*}
        |w^N\st|_{W^{-1,1}}^p
        \lesssim_{p,w_0}
        |t-s|^p + 
        \co_A(s,t) + \co_\natural(s,t) 
     \end{equation*}
     for all $(s,t)\in\Delta_T$ s.t. $\co_A(s,t)\le L$.

    At this point, \cref{lemma:compactnessArg} allows to select a subsequence (still denoted by $w^N$ for the sake of notation) converging to some $w\in C([0,T],W^{-1,1})$. 
    Testing \cref{eq:EulerApprox} against any $\phi \in W^{3,1}$, and passing to the limit $w^N \to w$ and $\mathbf{Z}^N\to\mathbf{Z}$, we can see that such $w$ is in fact a solution of \cref{eq:origEuler}.
    In particular, the convergence of the nonlinear transport term is ensured by the fact that $w^N \to w$ in $C([0,T],W^{-1,2})$ implies $u^N \to u$ in $C([0,T],L^2)$. 
    
    The proof is then complete once we notice that the remainder, defined by $w^{\natural}\st:=\lim_{N \to \infty}w^{N,\natural}\st$, has finite $\frac{p}{3}$-variation with respect to a suitable localization thanks to the uniform bound \cref{stimaRemUnif}.
\end{proof}


\section{Euler Lagrangian flows}\label{sec:ELflow}
This section is focused on the study of two Lagrangian equations associated with the flow of solutions to \cref{eq:origEuler}.
We are first going to prove that 
\[
\phi_t(x)=x+\int_0^tu(s,\phi_s(x))dr+\int_0^t\sigma_j(\phi_r(x))dZ^j_r,
\]
admits a solution, whenever $u$ is a spatial log-Lipschitz drift like the one obtained by applying the Biot-Savart kernel to a bounded function. Afterward, we are going to show a uniqueness result for the non-local RDE 
\[
\phi_t(x)=x+\int_0^t\int_{\T^2}K_{BS}(\phi_r(x)-\phi_r(y))w_0(y)dydr+\int_0^t\sigma_j(\phi_r(x))dZ^j_r,
\]
where $K_{BS}$ is the Biot-Savart kernel and $w_0$ is a function bounded in space and constant in time.


\subsection{Log-Lipschitz RDE}
Consider a bounded, measurable function $u:[0,T]\times \T^2\to \R^2$ that is log-Lipschitz in the spatial component, namely
\begin{equation*}
|u(t,x)-u(t,y)|\le L_u \gamma(|x-y|)
\qquad 
\forall t\in[0,T], x,y\in\T^2 ,
\end{equation*}
where we set $L_u>0$ and  $\gamma(r)=r(1-\log r)\uno_{(0,1/\e)}(r)+(r+1/e)\uno_{[1/e,\infty)}(r)$ for all $r\geq 0$.
We prove that the RDE
\begin{equation}\label{eq:eulerflowI}
\phi_t(x)=x+\int_0^tu(r,\phi_r(x))dr+\int_0^t\sigma_j(\phi_r(x))dZ^j_r,
\end{equation}
is well-posed by a limiting procedure based on the following a priori estimates.
Notice also that, in this section, we denote by $\ovsi$ the functional $\ovsi:C(\T^2,\R^2)\to L(\R^M,C(\T^2,\R^2))$ defined by $\ovsi(f)(z)\coloneqq\sum_{j=1}^M \sigma_j(f(x))z_j$, up to extending the $\sigma_j$ by periodicity. We will denote its Fr\'echet derivative by $D\ovsi$.

\begin{proposition}\label{AprioriEstimate}
	Let $u:[0,T]\times \T^2\to \R^2$ be bounded measurable, $\{\sigma_j\}_{j=1}^M$ be $C^3(\T^2,\R^2)$-vector fields, $\mathbf{Z}$ be a $p$-rough path and $y\in C(\T^2,\R^2)$. 
	Given a solution $(Y,\ovsi(Y))\in \mathscr{D}^{p\var}_{Z}(C(\T^2,\R^2))$ to 
	\begin{equation*}
	Y_t=y+\int_0^tu(r,Y_r)dr+\int_0^t\ovsi(Y_r)d\mathbf{Z}_r,
	\end{equation*}
	it holds that for $q\in(p,3)$, 
	\begin{align}
	\norm{Y}{q\var} & \le C_1 \\
	\norm{R^Y}{q\var,(\coL,\mu_0)} & \le C_2,
	\end{align}
	where $\coL(s,t)=\co_Z(s,t)+|t-s|^p$ and $C_1,C_2,\mu_0$ are positive constants depending only on $\|\ovsi\|_{C^3}$, $\|u\|_{L^{\infty}_{t,x}}$, $q$ and $\co_Z(0,T)$.
\end{proposition}

\begin{proof}
    Let us stress that in this proof we consider controls associated with the $q$- and $q/2$-variations with respect to the localization $(\coL,\mu)$ of several $1$- and $2$-index functions.

	If $(s,t)\in\Delta_T$ are such that $\coL(s,t)\le \mu$, then
	\begin{align*}
	|R^Y\st| 
	\le{} &
	\Big|\int_s^tu(r,Y_r)dr\Big| + \Big|\int_s^t\ovsi(Y_r)d\mathbf{Z}_r-\ovsi(Y_s)Z\st-D\ovsi(Y_s)\ovsi(Y_s)\Z\st\Big| \\
	& + |D\ovsi(Y_s)\ovsi(Y_s)\Z\st| \\
	\lesssim{} & 
	\norm{u}{L^{\infty}_{t,x}}|t-s| + \coLo_{R^{\ovsi(Y)}}(s,t)^{\frac{2}{q}}\co_Z(s,t)^{\frac{1}{q}} + \coLo_{\ovsi(Y)'}(s,t)^{\frac{1}{q}}\co_Z(s,t)^{\frac{2}{q}} \\
	&  + \norm{\ovsi}{C^1}^2 \co_Z(s,t)^{\frac{2}{q}} \\
	\lesssim{} & 
	\norm{u}{L^{\infty}_{t,x}}|t-s| + \norm{\ovsi}{C^2} \big(\coLo_Y(s,t)^{\frac{2}{q}}+\coLo_{R^Y}(s,t)^{\frac{2}{q}}\big) \co_Z(s,t)^{\frac{1}{q}} \\
	& + \norm{\ovsi}{C^2}^2\coLo_{Y}(s,t)^{\frac{1}{q}}\co_Z(s,t)^{\frac{2}{q}} + \norm{\ovsi}{C^1}^2 \co_Z(s,t)^{\frac{2}{q}} ,
	\end{align*}
	where the second inequality is due to the \cref{ContRouInt}, while the third one follows from standard facts (see \cite[Chapter 7]{FH20}) and the localization $(\overline{\omega},\mu)$.\\
	At this point we can apply \cref{BestControl} and, fixed $\mu$ small enough to have $\norm{\ovsi}{C^2}^{q/2} \co_Z(s,t)^{1/2} < 1/2$, absorb the control $\coLo_{R^Y}$ on the left-hand side: 
	\begin{align*}
	\coLo_{R^Y}(s,t)
	\lesssim_q{} & 
	2 \norm{u}{L^{\infty}_{t,x}}^{\frac{q}{2}} |t-s|^{\frac{q}{2}} + \coLo_Y(s,t) + \norm{\ovsi}{C^2}^{\frac{q}{2}} \coLo_Y(s,t)^{\frac{1}{2}} \co_Z(s,t)^{\frac{1}{2}} \\
	& + 2 \norm{\ovsi}{C^2}^q \co_Z(s,t) \\
	\lesssim_q{} & 
	\norm{u}{L^{\infty}_{t,x}}^{\frac{q}{2}} |t-s|^{\frac{q}{2}} + \coLo_Y(s,t) + \norm{\ovsi}{C^2}^q \co_Z(s,t) .
	\end{align*}
	Also, by definition of $(Y,\ovsi(Y))\in \mathscr{D}^{p\var}_{Z}$, 
	\begin{align*}
	|Y\st|
	\le{} &
	|\ovsi(Y_s)Z\st+R^Y\st| \\
	\le{} & 
	\norm{\ovsi}{C^0} \co_Z(s,t)^{\frac{1}{q}} + \coLo_{R^Y}(s,t)^{\frac{2}{q}} \\
	\lesssim_q &
	\norm{\ovsi}{C^0} \co_Z(s,t)^{\frac{1}{q}} + \norm{u}{L^{\infty}_{t,x}} |t-s| + \coLo_Y(s,t)^{\frac{2}{q}}
	+ \norm{\ovsi}{C^2}^2 \co_Z(s,t)^{\frac{2}{q}} ,
	\end{align*}
	for all $(s,t)\in \Delta_T$ such that $\coL(s,t)\le \mu$.
	Therefore, \cref{BestControl} implies that there exists a positive constant $C$, depending on $q$, $\norm{\ovsi}{C^2}$ and $\norm{u}{L^{\infty}_{t,x}}$, such that for all $(s,t)\in \Delta_T$
	\[
	\coLo_Y(s,t)\le C \mu (1 + \mu) + C \coLo_Y(s,t)^2 .
	\]
	Therefore, choosing $\mu_0 \leq \mu$ to have $4 C^2 \mu_0 (1+\mu_0) \leq 1$, given
	\[
	x_{1,2}^{\mu} 
	= 
	\frac{1}{2} \left( 1 \pm \sqrt{1 - 4 C^2 \mu (1+\mu)} \right) , 
	\qquad \mu \leq \mu_0,
	\]
	we find that 
	\[
	\coLo_Y(s,t)\le x_2^{\mu} 
	\quad \text{or} \quad 
	\coLo_Y(s,t)\geq x_1^{\mu}.
	\]
	
	We know by \cref{EstimateControl} that, for $\mu$ very small, the former inequality holds (indeed, the control goes to $0$ as $\mu$ does).
	At the cost of possibly re-choosing it, for a small enough $\mu_0>0$ we have $x_1^{\mu_0} > \frac{1}{2}$ and $x_2^{\mu_0}\le \frac{1}{8}$. To show that $\norm{Y}{q,[s,t],(\coL,\mu)}^q \le x_2^{\mu_0}$ for all $\mu \le \mu_0$ it is then sufficient to exclude any jump of size greater than 4. 
	To this aim, \cref{ControlloPerStimaApriori} yields that
	\begin{align*}
	& \norm{Y}{q,[s,t],(\coL,\mu)}^q\le 2^{q-1}\norm{Y}{q,[s,t],(\coL,\mu/2)}^q\le 2^{q-1}\lim_{\tilde{\mu} \to \mu^-}\norm{Y}{q,[s,t],(\coL,\tilde{\mu})}^q , 
	\end{align*}
	and since $ 4 x_2^{\mu_0} < x_1^{\mu_0}$ this implies that we can never go from $\coLo_Y(s,t)\le x_2^{\mu_0}$ to $\coLo_Y(s,t)\geq x_1^{\mu}$.
	
	At last, \cref{CompareVar} allows us to conclude the proof.
\end{proof}

We have proved that the localized variation of both $Y$ and its remainder are bounded from above whenever the localization we choose is small enough. Let us now turn to the proof that a solution to \cref{eq:eulerflowI} is continuous with respect to its initial datum, driver, drift and diffusion.

\begin{proposition}\label{ConvergenceEstimate}
	Let $\mathbf{Z}^1,\mathbf{Z}^2$ be two $p$-rough paths, $u_1,u_2:[0,T]\times \T^2 \to \R^2$ bounded, measurable and log-Lipschitz in the spatial component, $\ovsi_1$, $\ovsi_2$ be two family of $M$ divergence-free vector fields in $C^3(\T^2,\R^2)$ and $y_1,y_2 \in C(\T^2,\R^2)$. 
	Moreover, assume that $(Y^i_t,\ovsi_i(Y^i_t))\in \mathscr{D}^{p\var}_{Z^i}(C(\T^2,\R^2))$ are controlled rough paths such that
	\begin{equation*}
	Y^i_t=y_i+\int_0^tu_i(r,Y_r)dr+\int_0^t\ovsi_i(Y_r)dZ^i_r \quad \forall t \in [0,T], i=1,2 .
	\end{equation*}
	Then, there exists a positive constant $C$ depending on $\norm{\ovsi_i}{C^3}$, $\norm{u_i}{L^{\infty}_{t,x}}$, $\co_{Z^i}(0,T)$ and the log-Lipschitz constants $L_{u_i}$ such that
	\begin{multline}\label{MainEstimateLagrangian}
	\sup_{t\in[0,T]} \norm{Y^1-Y^2}{C(\T^2)} \\
	\le 
	C \Big( \norm{y_1-y_2}{C(\T^2)}  
    + \norm{\ovsi_1-\ovsi_2}{C^3}
	+ \co_{Z^1-Z^2}(0,T)^{\frac{1}{2q}} \\
	+ \norm{u_1-u_2}{L^{\infty}_{t,x}}T 
    + \int_0^T\gamma(\norm{Y^1_r-Y^2_r}{C(\T^2)})dr \Big) ,  
	\end{multline}
\end{proposition}

\begin{proof}
	We wish to estimate $\sup_{t\in[0,T]} \norm{Y^1-Y^2}{C(\T^2)}$ by means of \cref{Rough Gronwall}.
	Notice that what we really need to estimate is the difference $|R^1\st-R^2\st|$, since 
	\begin{equation}\label{IneqY1}
	\begin{aligned}
	|Y^1_{s,t}-Y^2_{s,t}|
	={} &
	|\ovsi_1(Y^1_s)Z^1\st-\ovsi_2(Y^2_s)Z^2\st+R^1\st-R^2\st| \\
	\le{} &
	\big( \norm{\ovsi_2}{C^1} I\st + \norm{\ovsi_1-\ovsi_2}{C^0} \big) |Z^2\st| + \norm{\ovsi_1}{C^0} |Z^1\st-Z^2\st| \\ 
    & + |R^1\st-R^2\st|,
	\end{aligned}
	\end{equation}
	where we called $I\st\coloneqq\sup_{r\in [s,t]}|Y^1_r-Y^2_r|$ to ease the notation and where $R^i$ is the Gubinelli remainder of $Y^i$, for $i=1,2$.
	
	In the following, we always consider controls associated with the $q$- and $q/2$-variations with respect to the localization $(\co_{Z^1}(s,t)+\co_{Z^2}(s,t)+|t-s|^p,\mu)$ of several $1$- and $2$-index functions. In particular, adding $|t-s|^p$ results in a strictly positive control. As in \cref{AprioriEstimate}, the parameter $\mu$ will be properly chosen.
	For the sake of readability, we will also denote the controls associated with the Gubinelli derivatives (respectively, remainders) of $\ovsi_1(Y^1)$ and $\ovsi_1(Y^1)-\ovsi_2(Y^2)$ by $\coLo_{D\ovsi_1}$ and $\coLo_{D(\ovsi_1-\ovsi_2)}$ (respectively, $\coLo_{R^{\ovsi_1}}$ and $\coLo_{R^{\ovsi_1}-R^{\ovsi_2}}$).
	
	For all $(s,t)\in \Delta_T$ such that $\co_{Z^1}(s,t)+\co_{Z^2}(s,t)+|t-s|^p\le \mu$,
	\begin{equation}\label{IneqR}
	\begin{aligned}
	|R^1\st-R^2\st|
	\le{} & 
	\Big|\int_s^tu_1(r,Y^1_r)-u
	_2(r,Y^2_r)dr\Big| \\
	& +\Big|\int_s^t\ovsi_1(Y^1_r)dZ^1_r-\ovsi_1(Y^1_s)Z^1\st-\int_s^t\ovsi_1(Y^2_r)dZ^2_r+\ovsi_2(Y^2_s)Z^2\st\Big|\\
	\lesssim{} & 
	\norm{u_1-u_2}{L^{\infty}_{t,x}}|t-s|+\min\{L_{u_1},L_{u_2}\}\int_s^t\gamma(|Y^1_r-Y^2_r|)dr\\
	& + \Big|\int_s^t\ovsi_1(Y^1_r)dZ^1_r-\ovsi_1(Y^1_s)Z^1\st-\int_s^t\ovsi_1(Y^2_r)dZ^2_r+\ovsi_2(Y^2_s)Z^2\st\Big| .
	\end{aligned}
	\end{equation}
	We are left with a rough integral term, that can be estimated by means of \cref{SewingLemma}. Indeed,
	\begin{equation}\label{IneqRoughTerm}
	\begin{aligned}
	\Big|\int_s^t\ovsi_1(Y^1_r) & dZ^1_r -\ovsi_1(Y^1_s)Z^1\st-\int_s^t\ovsi_1(Y^2_r)dZ^2_r+\ovsi_2(Y^2_s)Z^2\st\Big| \\
	& \lesssim_{\norm{\sigma_i}{C^3}} 
	\Big|D\ovsi_1(Y^1_s)\ovsi_1(Y^1_s)\textbf{Z}^1\st-D\ovsi_2(Y^2_s)\ovsi_2(Y^2_s)\textbf{Z}^2\st\Big| \\
	& \qquad + \coLo_{R^{\ovsi_1}}(s,t)^{\frac{2}{q}}\co_{Z^1-Z^2}(s,t)^{\frac{1}{q}}+\coLo_{R^{\ovsi_1}-R^{\ovsi_2}}(s,t)^{\frac{2}{q}}\co_{Z^1}(s,t)^{\frac{1}{q}} \\
	& \qquad + \co_{Z^1-Z^2}(s,t)^{\frac{2}{q}}\coLo_{D\ovsi_2}(s,t)^{\frac{1}{q}}+\coLo_{D(\ovsi_2-\ovsi_1)}(s,t)^{\frac{1}{q}}\co_{Z^1}(s,t)^{\frac{2}{q}} \\
	& \lesssim_{\norm{\sigma_i}{C^3}} 
	\co_{Z^1-Z^2}(s,t)^{\frac{2}{q}}+\co_{Z^1}(s,t)^{\frac{2}{q}}\big(I\st+|\ovsi_1-\ovsi_2|_{C^2}\big) \\
	& \qquad + \coLo_{R^{\ovsi_1}}(s,t)^{\frac{2}{q}}\co_{Z^1-Z^2}(s,t)^{\frac{1}{q}}+\coLo_{R^{\ovsi_1}-R^{\ovsi_2}}(s,t)^{\frac{2}{q}}\co_{Z^1}(s,t)^{\frac{1}{q}} \\
	& \qquad + \co_{Z^1-Z^2}(s,t)^{\frac{2}{q}}\coLo_{D\ovsi_2}(s,t)^{\frac{1}{q}}+\coLo_{D(\ovsi_2-\ovsi_1)}(s,t)^{\frac{1}{q}}\co_{Z^1}(s,t)^{\frac{2}{q}}.
	\end{aligned}
	\end{equation}
	
	We report below some estimates for those controls above that depend on $\ovsi_i$. Recall that multiplying a control by any power of $I_{s,t}$ yields a new control (see e.g. \cite[Chapter 1]{FV10}). 
	By combining \cref{BestControl} and a modification of \cite[Lemma 7.3]{FH20} in terms of $p$-variation, we have that
	\begin{equation}\label{IneqControls}
	\begin{aligned}
	\coLo_{D\ovsi_i}(s,t)
	\lesssim_{\ovsi_i,q}{} & 
	\coLo_{Y^i}(s,t) , \\
	\coLo_{D(\ovsi_2-\ovsi_1)}(s,t) 
	\lesssim_{\ovsi_i,q}{} & 
	\big[\coLo_{Y^1}(s,t) \norm{\ovsi_1-\ovsi_2}{C^2}^q+\coLo_{Y^1-Y^2}(s,t) \\
	& + I\st^q\big(\coLo_{Y^1}(s,t)+\coLo_{Y^2}(s,t)\big)\big] , \\
	\coLo_{R^{\ovsi_1}}(s,t)
	\lesssim_{\ovsi_i,q}{} & 
	\big(\coLo_{R^1}(s,t)+\coLo_{Y^1}(s,t)\big) , \\
	\coLo_{R^{\ovsi_1}-R^{\ovsi_2}}(s,t) 
	\lesssim_{\ovsi_i,q}{} & 
	\coLo_{R^1-R^2}(s,t) + \coLo_{R^1}(s,t) \big( \norm{\ovsi_1-\ovsi_2}{C^2}^{\frac{q}{2}} + I\st^{\frac{q}{2}} \big) \\
	& + \coLo_{Y^2}(s,t)\norm{\ovsi_1-\ovsi_2}{C^2}^{\frac{q}{2}}+\coLo_{Y^1-Y^2}(s,t) \\
	& + I\st^{\frac{q}{2}} \coLo_{Y^1}(s,t) .
	\end{aligned}    
	\end{equation}
	
	\cref{IneqR,IneqRoughTerm,IneqControls} above imply that there is a  positive constant $C$, depending on $\min\{L_{u_1},L_{u_2}\}$, $\max\{\norm{\ovsi_1}{C^3},\norm{\ovsi_2}{C^3}\}$ and $q$, such that
	\begin{align*}
	|R^1\st-R^2\st|
	\lesssim_C{} & 
	\|u_1-u_2\|_{L^{\infty}_{t,x}}|t-s| + \int_s^t\gamma(|Y^1_r-Y^2_r|)dr \\
	& + \co_{Z^1-Z^2}(s,t)^{\frac{2}{q}} + \co_{Z^1}(s,t)^{\frac{2}{q}} \big( I\st+\|\ovsi_1-\ovsi_2\|_{C^2} \big) \\
	& + \big( \coLo_{R^1}(s,t) + \coLo_{Y^1}(s,t) \big)^{\frac{2}{q}} \co_{Z^1-Z^2}(s,t)^{\frac{1}{q}} + \co_{Z^1}(s,t)^{\frac{1}{q}}\big[ \coLo_{R^1-R^2}(s,t) \\
	& + \coLo_{R^1}(s,t)\big(\|\ovsi_1-\ovsi_2\|_{C^2}^{\frac{q}{2}} +I\st^{\frac{q}{2}} \big) + \coLo_{Y^2}(s,t)\|\ovsi_1-\ovsi_2\|_{C^3}^{\frac{q}{2}} \\ 
	& + \coLo_{Y^1-Y^2}(s,t)+I\st^{\frac{q}{2}}\coLo_{Y^1}(s,t)\big]^{\frac{2}{q}} + \co_{Z^1-Z^2}(s,t)^{\frac{2}{q}} \coLo_{Y^2}(s,t)^{\frac{1}{q}} \\
	& +\co_{Z^1}(s,t)^{\frac{2}{q}}\big[\coLo_{Y^1}(s,t)\|\ovsi_1-\ovsi_2\|_{C^2}^q + \coLo_{Y^1-Y^2}(s,t) \\
	& + I\st^q\big(\coLo_{Y^1}(s,t) + \coLo_{Y^2}(s,t)\big) \big]^{\frac{1}{q}} ,
	\end{align*}
	for all $(s,t)\in \Delta_T$ such that $\co_{Z^1}(s,t)+\co_{Z^2}(s,t)+|t-s|\le \mu$. At this point one can fix $\mu$ small enough for the term $C \co_{Z^1}(s,t)^{\frac{1}{q}} \coLo_{R^1-R^2}(s,t)^{\frac{2}{q}}$ term to be smaller than $2^{-\frac{2}{q}}\coLo_{R^1-R^2}(s,t)^{\frac{2}{q}}$. \cref{BestControl} implies that in this case 
	\begin{equation}\label{Ineq2R}
	\begin{aligned}
	\coLo_{R^1-R^2}(s,t)
	\lesssim{} & 
	\left[ \norm{u_1-u_2}{L^{\infty}_{t,x}}|t-s|+\int_s^t\gamma(|Y^1_r-Y^2_r|)dr \right]^{\frac{q}{2}} \\
	& + \co_{Z^1-Z^2}(s,t) + \co_{Z^1}(s,t) \big( I\st + \norm{\ovsi_1-\ovsi_2}{C^1}\big)^{\frac{q}{2}} \\
	& + \big( \coLo_{R^1}(s,t) + \coLo_{Y^1}(s,t) \big) \co_{Z^1-Z^2}(s,t)^{\frac{1}{2}} \\
	& + \co_{Z^1}(s,t)^{\frac{1}{2}} \big[ \coLo_{R^1}(s,t) \big( \norm{\ovsi_1-\ovsi_2}{C^2}^{\frac{q}{2}} + I\st^{\frac{q}{2}} \big) + I\st^{\frac{q}{2}}\coLo_{Y^1}(s,t) \\
	& + \coLo_{Y^2}(s,t) \norm{\ovsi_1-\ovsi_2}{C^2}^{\frac{q}{2}} \big] + \coLo_{Y^2}(s,t)^{\frac{1}{2}} \co_{Z^1-Z^2}(s,t) \\
	& + \co_{Z^1}(s,t)\big[\coLo_{Y^1}(s,t)\|\ovsi_1-\ovsi_2\|_{C^2}^q \\
    & + \big(\coLo_{Y^1}(s,t) + \coLo_{Y^2}(s,t) \big) I\st^q \big]^{\frac{1}{2}} \\
	& + \co_{Z^1}(s,t)^{\frac{1}{2}} \coLo_{Y^1-Y^2}(s,t) + \co_{Z^1}(s,t) \coLo_{Y^1-Y^2}(s,t)^{\frac{1}{2}} .
	\end{aligned}
	\end{equation}
	Finally, we can combine \cref{Ineq2R,IneqY1,BestControl} to get 
	\begin{equation}\label{IneqY2}
	\begin{aligned}
	\coLo_{Y^1-Y^2}(s,t)
	\le{} & 
	\co_{Z^2}(s,t) \big( \norm{\ovsi_2}{C^1} I\st + \norm{\ovsi_1-\ovsi_2}{C^0} \big)^q \\
    & + \norm{\ovsi_1}{C^0}^q \co_{Z^1-Z^2}(s,t) + \coLo_{R^1-R^2}(s,t)^2 \\
	\le{} & 
	C\big[F(s,t)+\mu\hspace{2pt} \coLo_{Y^1-Y^2}(s,t)^2+\mu^2 \coLo_{Y^1-Y^2}(s,t)\big],
	\end{aligned}    
	\end{equation}
	where $F(s,t)$ contains all the terms apart from the ones involving $\coLo_{Y^1-Y^2}$. The latter inequality yields that
	\begin{equation*}
	\coLo_{Y^1-Y^2}(s,t)\le x_2^{\mu}(s,t) \quad or \quad \coLo_{Y^1-Y^2}(s,t)\geq x_1^{\mu}(s,t),
	\end{equation*}
	where 
	\[
	x_{1,2}^{\mu}(s,t) = \frac{1-C\mu^2\pm \sqrt{(1-C\mu^2)^2-4\mu C^2 F(s,t)}}{2\mu C} . 
	\]
	By \cref{AprioriEstimate}, $\norm{F}{L^{\infty}} \le M$. We can then choose $\mu$ sufficiently small to have
	\begin{equation*}
	x^{\mu}_1(s,t)\geq \frac{1-C\mu^2+\sqrt{(1-C\mu^2)^2-4\mu C^2 M}}{2\mu C}\geq 3M.
	\end{equation*}
	Notice in particular that the choice of $\mu$ only depends on the parameters of the problem, i.e. $\norm{\ovsi_i}{C^3}$, $\norm{u_i}{L^{\infty}_{t,x}}$, $L_{u_i}$ and $\co_{Z^i}(0,T)$. Moreover, again \cref{AprioriEstimate} implies that 
	\begin{multline}\label{IneqY3}
	\coLo_{Y^1-Y^2}(s,t) \\
    \le 
    x_2^{\mu}(s,t)=\frac{1-C\mu^2- \sqrt{(1-C\mu^2)^2-4\mu C^2 F(s,t)}}{2\mu C} \\
    \le 
    2 F(s,t),
	\end{multline}
	where the last inequality is due, up to shrinking $\mu$ again, to 
	\[
	\sqrt{a}-b \leq \sqrt{a-b}
	\qquad 
	\forall a>\frac{1}{4}, b\in [0,a-\frac{1}{4}] .
	\]
	At last, \cref{IneqY3} implies that 
	\begin{equation}\label{IneqY4}
	\begin{aligned}
	|(Y^1-Y^2)\st| 
	\lesssim{} &
	F(s,t)^{\frac{1}{q}} \\
	\le{} & 
	\norm{u_1-u_2}{L^{\infty}_{t,x}}|t-s| + \int_s^t\gamma(|Y^1_r-Y^2_r|)dr + C \co_{Z^1-Z^2}(s,t)^{\frac{1}{q}}\\
	& + C \big(\co_{Z^1}(s,t) + \co_{Z^2}(s,t)\big)^{\frac{1}{q}} \big( I\st + \norm{\ovsi_1-\ovsi_2}{C^3} \big)
	\end{aligned}
	\end{equation}
	for all $(s,t)\in \Delta_T$ such that $\co_{Z^1}(s,t)+\co_{Z^2}(s,t)+|t-s|^p\le \mu$.
	
	The above satisfies the hypothesis of the Rough Gr\"onwall Lemma (\cref{Rough Gronwall}), which we can then apply to conclude the proof.
\end{proof}

We are finally ready to state a well-posedness result for \cref{eq:eulerflowI}. We also prove that the solution is measure-preserving in space for all times $t\in[0,T]$.

\begin{theorem}\label{ExistenceLagr}
	Let $\mathbf{Z}$ be a $p$-geometric rough path, $p \in [2,3)$, $u:[0,T]\times \T^2 \to \R^2$ be a bounded, measurable and log-Lipschitz in space, $\{\sigma_j\}_{j=1}^M$ be $C^3(\T^2)$-vector fields and $\phi_0 \in C(\T^2,\R^2)$.
	
	Then, for $q\in(p,3)$ there exist a localization $(\coL,L)$ and a controlled rough path $(\phi_t,\ovsi(\phi_t))\in \mathscr{D}^{q\var}_{Z,(\coL,L)}(C(\T^2,\R^2))$ solving \cref{eq:eulerflowI}, namely
	\[
	\phi_t=\phi_0+\int_0^tu(r,\phi_r)dr+\int_0^t\ovsi(\phi_r)d\mathbf{Z}_r.
	\]
	Such $(\phi_t,\ovsi(\phi_t))$ is unique: if there exists $(\varphi_t,\ovsi(\varphi_t))\in \mathscr{D}^{q\var}_{Z,(\coL',L')}(C(\T^2,\R^2))$  which also solves it, then $\varphi_t=\phi_t$ for all $t$.
	
	In addition, if $\phi_0=Id_{\T^2}$ is measure preserving and $u$, $\sigma_j$ are divergence-free, then $\phi_t$ is also measure preserving for all $t \in [0,T]$.
\end{theorem}

\begin{proof}
	Since $\mathbf{Z}$ is geometric, there is a sequence of functions in $C^{\infty}([0,T],\R^M)$ whose canonical lifts converge to $\mathbf{Z}$ in $\rpspace$. In particular, we choose a subsequence $\{Z^N\}_{N\in\N}$ such that, $\co_{Z^N-Z}(0,T)<\frac{1}{N}$ for $N\to\infty$. Without loss of generality, we can also assume that 
	\[
	|Z^N\st| \le \co_Z(s,t)^{\frac{1}{p}} 
	\text{ and } 
	|\Z^N\st|\le\co_Z(s,t)^{\frac{2}{p}}.
	\]
	In addition, let $\rho$ be a mollifier (in space) and define, up to properly extending $u$ to $\R^2$ by periodicity, $u^N=u * \rho_{\frac{1}{N}}$.
	
	Since $u^N$ is sufficiently smooth for all $N\in \N$, standard results (see e.g. \cite{FH20}) imply the existence of a unique $\phi^N \in C^{p\var}$ such that $(\phi^N,\ovsi(\phi^N)) \in \mathscr{D}^{p\var}_{Z^N}$ solves the RDE 
	\begin{equation*}
	\phi^N_t=\phi_0+\int_0^t u^N(r,\phi^N_r)dr+\int_0^t\ovsi(\phi^N_r)dZ^N_r.
	\end{equation*}
	\cref{ConvergenceEstimate} and a comparison argument imply that,
	\begin{equation}\label{eq:LinfCauchyFlow}
	|\phi_t^N-\phi_t^{\tilde{N}}|
	\lesssim 
	 z(t,\norm{u^N-u^{\tilde{N}}}{L^{\infty}_{t,x}}T+\omega_{Z^{N}-Z^{\tilde{N}}}(0,T)^{\frac{1}{2q}}),
	\end{equation}
	where $z(t,z_0)$ is the solution to $\dot{z_t}=\gamma(z_t)$ with $z_0$ as starting condition. 
	Notice that $z(t,z_0)\le ez_0^{exp(-t)}$ as long as $z_0$ is sufficiently small (cf. \cite{BFM16}).
	
	By combining \cref{AprioriEstimate}, \cref{eq:LinfCauchyFlow} and the interpolation inequality
	\[
	\norm{\phi^N-\phi^{\tilde{N}}}{q,[0,T]}
	\lesssim 
	\big( \norm{\phi^N-\phi^{\tilde{N}}}{p,[0,T]} \big)^{\frac{p}{q}} \big( \norm{\phi^N-\phi^{\tilde{N}}}{L^{\infty}} \big)^{1-\frac{p}{q}} ,
	\]
	we eventually see that $\{\phi^N\}$ and $\{\ovsi(\phi^N)\}$ are Cauchy sequences in $C^{q\var}$. 
	Moreover, we know that $\norm{u^N}{L^{\infty}_{t,x}} \lesssim \norm{u}{L^{\infty}_{t,x}}$ and that the sequence $u^N$ converges uniformly in $(t,x)$ to $u$ and is equi-log-Lipschitz, i.e.
	\begin{equation*}
	|u^N(t,x)-u^N(t,y)|
	\le 
	L_u \gamma(|x-y|) 
	\qquad 
	\forall N\in \N, t\in [0,T] \text{ and } x,y\in \R^2.
	\end{equation*}
	These properties, \cref{AprioriEstimate,ConvergenceEstimate} and \cref{Ineq2R} imply that there exists $\mu_0>0$ such that $\norm{R^N-R^{\tilde{N}}}{\frac{q}{2},(\coL,\mu_0),[0,T]}$ is a Cauchy sequence, with $R^N$ being the remainder of $\phi^N$ with respect to $Z^N$ and $\coL(s,t)=\omega_Z(s,t)+|t-s|^p$. 
	
	It follows that, given $\lim_{N\to \infty}(\phi^N, \ovsi(\phi^N)) \eqqcolon (\phi,\ovsi(\phi)) \in \mathscr{D}^{q\var}_{Z,(\coL,\mu_0)}(C_b(\T^2,\R^2))$, by continuity of rough integrals, i.e. \cref{ContRouInt}, $(\phi, \ovsi(\phi))$ solves 
	\begin{equation*}
	\phi_t=\phi_0+\int_0^t u(s,\phi_s)dx+\int_0^t\ovsi(\phi_s)dZ_s.
	\end{equation*}
	
	Uniqueness then comes from a trivial generalization of \cref{ConvergenceEstimate}, allowing it to hold for localized controlled rough paths as well.
	
	As for measure-preservingness of  $\phi_t$, is given by \cref{eq:LinfCauchyFlow}. 
\end{proof}

Once the well-posedness for the RDE with log-Lipschitz drift is obtained, we conclude this section by studying the inverse flow map.

\begin{theorem}\label{thm:h_t}
	Let $\mathbf{Z}$, $p$, $u$ and $\ovsi$ be as in the previous theorem. Given $\phi_0 = Id$, if $\phi: [0,T]\times \T^2 \to \T^2$ is the solution of   
	\[
	\phi_t^x 
	= 
	x + \int_0^t u(r,\phi_r^x) dr - \int_0^t \ovsi(\phi_r^x) d\mathbf{Z}_r ,
	\]
	the following hold:
	\begin{enumerate}
		\item $\phi_t$ is invertible, for all $t \in [0,T]$;
		\item given $\psi \in C^{\infty}(\T^2,\R)$ and $h_t\coloneqq\psi (\phi_t^{-1})$, there exist a localization $(\coL, L)$ and $h^{\natural} \in  C_{2,\coL,L}^{p/3\var}(W^{-3,1})$ such that, for all $\gamma \in W^{3,1}$,		
		\begin{equation}
		h\st(\gamma) 
		=
		\int_s^t h_r u_r \sg \gamma dr + h_s \left( \left[ A^{1,*}\st + A^{2,*}\st \right] \gamma \right) + h^{\natural}\st (\gamma),
		\end{equation}
		where $(A^1,A^2)$ is the unbounded rough driver given by \cref{eq:urd1}.
	\end{enumerate}
\end{theorem}

\begin{proof}
	\begin{enumerate}
		\item Given $t \in [0,T]$, we can build the inverse of $\phi_t$ by solving the  equation backward in time. 
		Namely, $\phi_t^{-1} = Y_t$, where $(Y_s)_{s \in [0,t]}$ is the solution of
		\begin{equation*}
		Y_s = Id_{\T^2} - \int_0^s u(t-r,Y_r)dr - \int_0^s \ovsi(Y_r) d \overline{\mathbf{Z}}_r,
		\end{equation*}
		driven by the geometric rough path associated with $\overline{Z}_s=Z_{t-s}$ (we refer to \cite{FH20} for more details).
		\item Similarly to the previous theorem, consider two sequences $\{Z^N\}_{N \in \N}\subset C^{\infty}([0,T],\R^M)$, approximating $\mathbf{Z}$ in $\rpspace$, and $\{u_N\}_{N\in\N} \subset C^{\infty}([0,T]\times \T^2, \R^2)$, converging to $u$.
		
		For each $N$, the corresponding flows $\phi^N$ are invertible and sufficiently regular. 
		In particular since $\partial_t\big[ \phi_t^N(\phi_t^{N,-1}(x))\big]\equiv 0$, one has that
        \begin{equation*}
            \begin{aligned}
                \partial_t \big(\phi_t^{N,-1}(x)\big)&=-\big[(\grad^T\phi^N_t\big)^{-1}(\phi_t^{N,-1}(x))\big] \big[(\partial_t \phi^N_t)(\phi^{N,-1}_t(x))\big]\\
                &=-\big[(\grad^T\phi^N_t\big)^{-1}(\phi_t^{N,-1}(x))\big]\big[u_t^N(x)-\ovsi(x)\dot{Z}_t^N\big].
            \end{aligned}
        \end{equation*}
  Therefore,
		\begin{equation*}
		\partial_t h^N_t(x) 
		= -u^N_t(x)\sg h^N_t(x) + \ovsi(x)\sg h^N_t(x) \dot{Z}^N_t ,
		\end{equation*}
		with $h^N_t\coloneqq\psi( \phi^{N,-1}_t)$. 
		The result then follows as in \cref{th:vanishvisc}, namely we pass through the integral formulation and iterate the equation as we did in \cref{sec:formEq}. 
		Eventually, thanks to the convergence of $\phi_t^{N,-1}$ to $\phi_t^{-1}$ (by \cref{ConvergenceEstimate}), we can send $N$ to infinity in the equation solved by $h^N$ .
	\end{enumerate}
\end{proof}

\subsection{RDEs with non-local drift}
Let us consider the second Lagrangian RDE, which plays a crucial role in the proof of uniqueness for solutions of the Euler equation \eqref{eq:rougheuler}.
In particular, we prove a uniqueness result for the following RDE
\begin{equation}\label{eq:eulerFlowII}
\phi_t(x)=x+\int_0^t\int_{\T^2}K_{BS}(\phi_r(x)-\phi_r(y))w_0(y)dydr+\int_0^t\ovsi(\phi_r(x))d\mathbf{Z}_r,
\end{equation}
where $w_0$ is a bounded function.\\

\begin{theorem}\label{thm:uniqFlowII}
	Let $\mathbf{Z}$ be a $p$-geometric rough path, $p \in [2,3)$, $w_0\in L^{\infty}(\T^2)$,  $\{\sigma_j\}_{j=1}^M$ be $C^3(\T^2)$-vector fields and $\phi_0 \in C(\T^2,\R^2)$.
	
	If there exist two controls, $\coL_1$ and $\coL_2$, and two positive constants, $L_1$ and $L_2$, such that $(X,\ovsi(X))\in \mathscr{D}_{Z,(\coL_1,L_1)}^{p\var}(C(\T^2,\R^2))$ and $(Y,\ovsi(Y))\in \mathscr{D}_{Z,(\coL_2,L_2)}^{p\var}(C(\T^2,\R^2))$ are measure-preserving solutions to \cref{eq:eulerFlowII}, then
	\[
	X_t = Y_t 
	\qquad \forall t\in[0,T] .
	\]
\end{theorem}

\begin{remark}
	An existence result for \cref{eq:eulerFlowII} can be easily produced by combining \cref{ConvergenceEstimate}, \cref{ExistenceLagr} and the iterative procedure proposed in \cite{BFM16}. We avoid replicating such method, as it plays no role in the uniqueness of solutions to \cref{eq:rougheuler}. 
\end{remark}

\begin{proof}
	Let $x\in \T^2$. Then,
	\begin{multline*}
	X_t(x)-Y_t(x)
	=
	\int_0^t \int_{\T^2} \Big[ \big( K_{BS}(X_r(x)-X_r(y) \big) - \big( K_{BS}(Y_r(x)-Y_r(y) \big) \Big] w_0(y) dy dr \\
	+\int_0^t \ovsi(X_r(x))-\ovsi(Y_r(x))d\mathbf{Z}_r.
	\end{multline*}
	By following the proof of \cref{ConvergenceEstimate}, we end up with
	\begin{multline*}
	|X_t(x)-Y_t(x)|
	\lesssim 
	\norm{w_0}{L^{\infty}} \\
	\int_0^t \int_{\T^2} \Big|\big(K_{BS}(X_r(x)-X_r(y)\big) - \big(K_{BS}(Y_r(x)-Y_r(y)\big)\Big|dydr.
	\end{multline*}
	
	Now, recalling that the Biot-Savart kernel enjoys the property
	\[
	\int_{\T^2}|K_{BS}(x-y)-K_{BS}(x'-y)|dy\lesssim \gamma(|x-x'|),\quad \forall x,x' \in \T^2,
	\]
	we can integrate in space and exploit $\gamma$ being concave to obtain that
	\[
	\int_{\T^2}|X_t(x)-Y_t(x)|dx
	\lesssim 
	\norm{w_0}{L^{\infty}}\int_0^t \gamma\Big(\int_{\T^2}\Big|X_r(x)-Y_r(x)\Big|dx\Big)dr .
	\]
	At this point, as in \cref{ExistenceLagr}, a comparison argument yields the thesis.
\end{proof}


\section{Uniqueness of the solution}\label{sec:uniq}
In this section we exploit the uniqueness results obtained for the Lagrangian flows to prove that there exists a unique solution for the two-dimensional Euler equation \eqref{eq:origEuler}. 
To this aim, we show that $w_t=(\phi_t)_\# w_0$ following Dobrushin's strategy \cite{D79}, i.e. proving that 
\begin{equation}\label{eq:DubStyle}
\int_{\T^2}w_t(x)h_t(x)dx=\int_{\T^2}w_0(x)h_0(x)dx,
\end{equation}
where $h_t\coloneqq\psi (\phi^{-1}_t(x))$.\\
Since $w^{\natural}$ and $h^{\natural}$ are poorly regular, we cannot test directly the equation for $w_t$ against the one satisfied by $h_t$.
Therefore, we must develop a strategy similar to the one of \cite{Hof2,DEYA20193577}. 
There, a doubling variable technique is used to estimate some crucial quantities, such as the energy of the solution. 
Instead, we need to apply it for any smooth function $\psi$ in order to prove that any solution $w$ is transported by its corresponding flow.

\subsection{Doubling variable}
Let $(\sigma_j)_{j=1}^M\subset C^3(\T^2,\R^2)$ be a family of divergence-free vector fields, $w_0\in L^{\infty}(\T^2,\R)$ and $\psi \in C^{\infty}(\T^2,\R)$. Consider, according to \cref{def:roughsolution}, a solution $w$ to \cref{eq:rougheuler} with initial condition $w_0$ and consider $\phi$, $h$ as in \cref{thm:h_t}: $\phi$ solving
\begin{equation}\label{eq:flowSegniGiusti}
\phi_t=I_{\T^2}+\int_0^tu(r,\phi_r)dr-\int_0^t\ovsi(\phi_r)d\mathbf{Z}_r,
\end{equation}
where $u=K_{BS} \ast w$, and $h=\psi (\phi^{-1})$. 

Let us first state the following (cf. \cite{Hof1,Hof2, DEYA20193577}). 
From now on, we are also going to denote the Sobolev space $W^{n,2}(\T^2 \times \T^2)$ by $W^{n,2}_{\otimes}$.
\begin{lemma}\label{lemma:g}
	The mapping $g:[0,T]\to W^{0,2}_{\otimes}$ defined by $g_s(x,y)=w_s(x) \otimes h_s(y)$ satisfies the equation
	\begin{equation}\label{eq:doubleRPDE}
	g\st
	=
	-\int_s^t \left( u_r \sg w_r \otimes h_r + w_r \otimes u_r \sg h_r \right) dr + \Gamma^1\st g_s+\Gamma^2\st g_s+g^{\natural}\st,
	\end{equation}
	where $g^{\natural} \in C_{2,\coL,L}^{p/3\var}(W^{-3,\infty}_{\otimes})$ and $(\Gamma^1,\Gamma^2)$ is an unbounded rough driver on the scale $\{W^{n,2}_{\otimes}\}_n$ (see \cref{sec:URD&SO})  defined by 
	\begin{equation}\label{eq:urd2}
	\Gamma^1\st=A^1\st\otimes I +I\otimes A^1\st, \qquad \Gamma^2\st=A^2\st\otimes I+I\otimes A^2\st+A^1\st\otimes A^1\st,
	\end{equation}
	with $A^{i}$, for $i=1,2$, given by \cref{eq:urd1}.
\end{lemma}

In our setting, the above result follows directly from \cref{thm:h_t}.

Now, we want to test \cref{eq:doubleRPDE} against a smooth approximation of the diagonal $\delta_{x=y}$. Following \cite{Hof2}, let us denote $x_{\pm}\coloneqq\frac{x\pm y}{2}$ and define, for $n \in N$, the space
\begin{multline*}
\mathcal{E}^{n,\infty} 
\coloneqq 
\Big\{ \alpha\in W^{n,\infty}(\R^2\times \R^2) |\alpha(x+2k\pi,y+2k\pi)=\alpha(x,y) \ \forall k\in\Z^2 , \\
|x_{-}|>1 \implies \alpha(x,y)=0 \Big\} 
\end{multline*}
normed by 
\begin{equation*}
\norm{\alpha}{n,\nabla} 
\coloneqq 
\max_{k+l\le n}|\grad_{+}^k\grad_{-}^l\alpha(x,y)|_{L^{\infty}_{x,y}},
\end{equation*}
where $\grad_{\pm}\coloneqq\frac{1}{2}(\grad_{x}\pm \grad_{y})$.\\
We also need to introduce the dual spaces $\mathcal{E}^{-n,\infty} \coloneqq \big( \mathcal{E}^{n,\infty} \big)^*$, the duality pairing 
\begin{equation*}
\langle \gamma, \alpha \rangle_{\grad} 
\coloneqq 
\int_{B(0,1)} \int_{\T^2} \gamma (x_{+}+x_{-}, x_{+}-x_{-}) \alpha (x_{+}+x_{-}, x_{+}-x_{-}) d x_{+} d x_{-} 
\end{equation*}
and a blow-up transformation: for $\epsilon>0$ and $\alpha\in \mathcal{E}^{n,\infty}$, it is given by the map 
\begin{equation*}
T_{\epsilon} \alpha (x,y) 
\coloneqq 
\frac{1}{\epsilon^2} \alpha \Big( x_{+} + \frac{x_{-}}{\epsilon}, x_{+} - \frac{x_{-}}{\epsilon} \Big) ,
\end{equation*}
whose adjoint operator is $T_{\epsilon}^*\gamma (x,y) = \gamma (x_{+}+\epsilon x_{-}, x_{+}-\epsilon x_{-})$.

At this point, up to extending by periodicity, we can test $g\st$ against $T_{\epsilon}\alpha$. Let us start with the drift term.

\begin{lemma}\label{driftRPDEBlowup}
	\begin{enumerate}
		\item Let $\alpha \in \mathcal{E}^{1,\infty}$ and $\nu\st\in W^{-1,2}_{\otimes}$ be defined as 
		\[
		\nu\st 
		\coloneqq 
		-\int_s^t u_r \sg w_r \otimes h_r + w_r \otimes u_r \sg h_r dr . 
		\]
		Then 
		\begin{equation}\label{eq:driftUnifBound}
		|\nu \st(T_{\epsilon}\alpha)|
		\leq 
		C \norm{\alpha}{1,\grad}|t-s| , 
		\end{equation}
		where $C$ is a constant only depending on $\norm{w}{L^{\infty}_{t,x}}$ and $\norm{\psi}{L^{\infty}}$.
		\item Let $\alpha(x,y) = \beta(x_{+}) f(|x_{-}|^2)$, with $\beta \in C^{\infty}(\T^2,\R)$ and $f\in C^\infty(\R,\R^+)$  such that it has support contained into $[0,1]$ and $\int_{B(0,1)} f(|x_-|^2)dx_-=2$. Then,
		\begin{equation}\label{eq:driftLimit}
		\lim_{\epsilon\to 0}\nu \st(T_{\epsilon}\alpha)=\int_s^t\int_{\T^2}w_r(x)h_r(x)u_r(x)\sg \beta(x)dx.
		\end{equation}
	\end{enumerate}
\end{lemma}

\begin{proof}
	\begin{enumerate}
		\item We have that
		\begin{align*}
		\nu\st(T_{\epsilon}\alpha)&=\int_s^t dr \int \frac{1}{\epsilon^2}u_r(x)w_r(x)h_r(y) \grad_x \Big[\alpha\Big(x_{+}+\frac{x_{-}}{\epsilon},x_{+}-\frac{x_{-}}{\epsilon}\Big)\Big]dxdy\\
		&+\int_s^t dr\int \frac{1}{\epsilon^2}u_r(y)w_r(x)h_r(y) \grad_y\Big[\alpha\Big(x_{+}+\frac{x_{-}}{\epsilon},x_{+}-\frac{x_{-}}{\epsilon}\Big)\Big]dxdy\\
		&=\frac{1}{2}\int_s^t dr\int w_r(x_{+}+\epsilon x_{-})h_r(x_{+}-\epsilon x_{-})u_r^{\epsilon,+}(x_{+},x_{-})\cdot
			\big(\grad_{+}\alpha\big) (x,y)dx_{+}dx_{-}\\
		&+\frac{1}{2}\int_s^t dr\int w_r(x_{+}+\epsilon x_{-})h_r(x_{+}-\epsilon x_{-})\frac{1}{\epsilon}u_r^{\epsilon,-}(x_{+},x_{-})\cdot
			\big(\grad_{-}\alpha\big) (x,y)dx_{+}dx_{-}\\
        &\eqqcolon \nu \st(T_{\epsilon}\alpha)_++\nu \st(T_{\epsilon}\alpha)_- ,
		\end{align*}
		where $u^{\epsilon,\pm}_r(x_{+},x_{-})\coloneqq\Big[u_r(x_{+}+\epsilon x_{-})\pm u_r(x_{+}-\epsilon x_{-})\Big]$.\\
		The first term can be easily estimated uniformly in $\epsilon$. Instead, for the second term
		\begin{multline*}
		\Big|\int_s^t dr \int_{B(0,1)}dx_{-}\int_{\T^2}dx_{+} w_r(x_{+}+\epsilon x_{-}) h_r(x_{+}-\epsilon x_{-})\frac{1}{\epsilon}u_r^{\epsilon,-}(x_{+},x_{-})\cdot
		\big(\grad_{-}\alpha\big) (x,y)\Big| \\
        \le 
		\|w\|_{L^{\infty}_{t,x}}\|\psi\|_{L^{\infty}}\|\alpha\|_{1,\grad}\int_s^t dr\int_{B(0,1)}dx_{-}\int_{\T^2}dx_{+}\Big| \frac{u^{\epsilon,-}_r(x_{+},x_{-})}{\epsilon}\Big| \\
        \le{} 
        \|w\|_{L^{\infty}_{t,x}}\|\psi\|_{L^{\infty}}\|\alpha\|_{1,\grad}\|u\|_{L^{\infty}_t W^{1,1}_x}|t-s|.
		\end{multline*}
		\item It is enough to prove that $\lim_{\epsilon\to 0}\nu \st(T_{\epsilon}\alpha)_-=0$. In particular, up to approximating the velocity $u$ in $L^{\infty}_T W^{1,1}$ by smooth functions, let us assume that $u$ is smooth in space. \\
		Then,
		\begin{multline*}
		\nu \st(T_{\epsilon}\alpha)_-
        =  
        \int_s^t dr \int_{B(0,1)}dx_-\int_{\T^2}dx_{+}w_r(x_++\epsilon x_-)h_r(x_+-\epsilon x_-)\beta(x_+) \\
		\Big(\int_0^1 \grad^Tu_r(x_+-\epsilon x_-+2\theta \epsilon x_-)x_- d\theta\Big)\cdot x_-f'(|x_-|^2) ,
		\end{multline*}
		which, as $\epsilon$ goes to $0$, converges to 
		\begin{multline*}
		\int_s^tdr\int_{\T^2}dx_+w_r(x_+)h_r(x_+)\beta (x_+)\sum_{i,j=1}^2\partial_ju^i_r(x_+)\int_{B(0,1)}x_-^jx_-^if'(|x_-|^2)dx_- \\
		= 
		\int_s^tdr\int_{\T^2}dx_+w_r(x_+)h_r(x_+)\beta (x_+) \text{div}(u_r)(x_+) \int_0^1\pi\rho^3f'(\rho^2)d\rho 
		= 
		0
		\end{multline*}
	\end{enumerate}
\end{proof}
It is now the turn of the rough terms to be tested against $T_{\epsilon}\alpha$. The following has already been proved in \cite{DEYA20193577}, hence we are only going to give the statement. 

\begin{lemma}\label{diffusionRPDEBlowup}
	Consider the unbounded rough drivers $(\Gamma^1\st,\Gamma^2\st)$ defined in \cref{eq:urd2}:  $\big( \Gamma^{1,\epsilon}\st, \Gamma^{2,\epsilon}\st \big)_{\epsilon} \coloneqq \big( T_{\epsilon}^* \Gamma^1\st T_{\epsilon}^{*,-1}, T_{\epsilon}^* \Gamma^2\st T_{\epsilon}^{*,-1} \big)_{\epsilon}$ is a bounded family of unbounded rough drivers on the scale $\mathcal{E}^{n,\grad}$.
 
	Moreover, if $\alpha(x,y) = \beta(x_{+}) f(|x_{-}|^2)$, with $\beta \in C^{\infty}(\T^2,\R)$ and $f\in C^\infty(\R,\R^+)$  such that it has support contained into $[0,1]$ and $\int_{B(0,1)} f(|x_-|^2)dx_-=2$, then 
	\begin{equation*}
	\lim_{\epsilon\to 0} \langle T_{\epsilon}^*g_s, \Gamma^{i,\epsilon,*}\st \alpha\rangle_{\grad}=\int_{\T^2}w_s(x)h_s(x)(A^{i,*}\st \beta)(x)dx, 
	\qquad 
	i = 1,2, 
	\end{equation*}
	where $A^i$ is as in \cref{eq:urd1}.
\end{lemma}

We conclude the paragraph by identifying the equation satisfied by $T_{\epsilon}^*g_s$.
In order to estimate the remainder appearing in it, we first report a result from \cite{Hof2}.

\begin{lemma}\label{smoothing2}
	There exists a family of smoothing operators $(\overline{J}^{\eta})_{\eta \in [0,1]}$ on the scale $\mathcal{E}^{n,\grad}$.
\end{lemma}

Eventually, the following is what we need.

\begin{proposition}\label{prop:doubleRPDEBlowup}
	For any $\epsilon>0$, given $g$ as in \cref{lemma:g}, $g^{\epsilon}_s\coloneqq T^*_{\epsilon}g_s$ satisfies in $\mathcal{E}^{-3,\grad}$ the equation 
	\begin{equation}\label{eq:doubleRPDEBlowup}
	g^{\epsilon}\st=\nu^{\epsilon}\st+\Gamma^{1,\epsilon}\st g^{\epsilon}_s+
	\Gamma^{2,\epsilon}\st g^{\epsilon}_s+g^{\natural,\epsilon}\st,
	\end{equation}
	where $\nu^{\epsilon}\st(\alpha) = \nu\st(T_{\epsilon}\alpha)$ is defined as in \cref{driftRPDEBlowup}. 
	Moreover, the family $(g^{\epsilon,\natural}\st)_{\epsilon>0}$ is bounded in $ C_{2,\coL,L}^{p/3\var} (\mathcal{E}^{-3,\grad})$.
\end{proposition}
\begin{proof}
	The equation satisfied by $g^{\epsilon}_s$ can be easily found by testing \cref{eq:doubleRPDE} against $T_{\epsilon}\alpha$ for all $\alpha \in \mathcal{E}^{3,\grad}$.
	
	Instead, the uniform bound on the localized $\frac{p}{3}$-variation of the remainders $g^{\natural,\epsilon}$ follows from \cref{driftRPDEBlowup}, \cref{diffusionRPDEBlowup} and \cref{smoothing2} with a procedure analogous to the one already showed in \cref{lemma:apriori}. 
    Indeed, as shown in \cite{DEYA20193577}, such procedure can be replicated once we have, like here, a family of smoothing operators, the drifts $\{\nu^{\epsilon}\st\}_{\epsilon>0}$ uniformly in $C^{1\var}(\mathcal{E}^{-1,\grad})$ and $\{(\Gamma^{1,\epsilon}\st,\Gamma^{2,\epsilon}\st)\}_{\epsilon>0}$ uniformly unbounded rough drivers.
\end{proof}

\subsection{Uniqueness result}

We are finally ready to prove the remaining part of \cref{mainthm}, namely that a 2D Euler equation, driven by rough transport noise, has a unique solution whenever the initial condition is bounded.

\begin{theorem} \label{thm:uniq}
	Let $w_0\in L^{\infty}(\T^2)$, $\mathbf{Z}$ be a $p$-geometric rough path, $p \in [2,3)$, and $\{\sigma_j\}_{j=1}^M \subset C^3(\T^2,\R^2)$ be a family of divergence-free vector fields. 
	Given $w$ a solution of \cref{eq:origEuler} in the sense of \cref{def:roughsolution}, then $w$ is unique.
	
	In particular, it can be written as the advection of $w_0$ along the corresponding flow,
	\begin{equation*}
	w_t=(\phi_t)_{\#}w_0,
	\end{equation*}
	where $\phi$ solves \cref{eq:flowSegniGiusti}.
\end{theorem}
\begin{proof}
We start proving that any solution $w$ in the sense of \cref{def:roughsolution} is advected by the associated flow.

Consider to this goal $\psi \in C^{\infty}(\T^2,\R)$, $h_t=\psi(\phi_t^{-1})$ and $g^{\epsilon}_s=T_{\epsilon}^*(w_s\otimes h_s)$. 
Given $\alpha(x,y)=\beta(x_{+})f(|x_{-}|^2)$, with $\beta \in C^{\infty}(\T^2,\R)$ and $f\in C^\infty(\R,\R^+)$  with support contained into $[0,1]$ and $\int_{B(0,1)} f(|x_-|^2)dx_-=2$, we can test  \cref{eq:doubleRPDEBlowup} against $\alpha$ and send $\epsilon$ to $0$. 
Then, \cref{driftRPDEBlowup,diffusionRPDEBlowup} imply that
\begin{multline*}
	\int_{\T^2} \left( w_t(x) h_t(x) - w_s(x) h_s(x) \right) \beta(x) dx = \int_s^t \int_{\T^2} w_r(x) h_r(x) u_r(x)\sg \beta(x) dx dr \\ 
	+ \int_{\T^2} w_s(x) h_s(x) A^{1,*}\st \beta(x) dx + \int_{\T^2} w_s(x) h_s(x) A^{2,*}\st \beta(x) dx + \overline{g}^{\natural}\st(\beta),
\end{multline*}
where $\overline{g}^{\natural}\st(\beta)\coloneqq\lim_{\epsilon\to 0}g^{\epsilon,\natural}\st(\alpha)$ has finite localized $\frac{p}{3}$-variation in $W^{-3,\infty}$ thanks to \cref{prop:doubleRPDEBlowup}. In particular, if we choose $\beta\equiv 1$, we get the equality 
\begin{equation*}
	\int_{\T^2} \left( w_t(x) h_t(x) - w_s(x) h_s(x) \right) dx = \overline{g}^{\natural}\st (1) ,
\end{equation*}
which yields that $\overline{g}^{\natural}\st(1)$ is an additive function with finite $\frac{p}{3}$-variation, meaning that it is identically $0$.

Hence, up to approximation with smooth functions, we can take $\psi=\overline{\psi} \circ \phi_t$ for any $\overline{\psi}\in C^{\infty}(\T^2,\R)$ and obtain that
\begin{equation}\label{eq:flowuniq}
	\int_{\T^2}w_t(x)\overline{\psi}(x)dx=\int_{\T^2}w_0(x)\overline{\psi}(\phi_t(x))dx,
\end{equation}
which is, $w_t$ is transported along the associated flow.
		
Consider now two solutions $w^1,w^2$ to \cref{eq:origEuler}, both starting at $w_0$, driven by $\mathbf{Z}$ and the same vector fields $(\sigma_j)_{j=1}^M$. We have proved that $w^i_t=(\phi_t^{i})_{\#}w_0$, $i=1,2$, where $\phi^i_t$ solves \cref{eq:eulerflowI} with drift $K_{BS}\ast w^i_t$.
Noting that
\begin{align*}
	\int_0^t(K_{BS}\ast w^i_r)(\phi^i_r(x))dr
	& = \int_0^t\int_{\T^2}K_{BS}(\phi^i_r(x)-y)w_r^i(y)dydr \\
	& = \int_0^t \int_{\T^2}K_{BS}(\phi^i_r(x)-\phi^i_r(y))w_0(y)dydr,
\end{align*}
hence $\phi^1_t$ and $\phi^2_t$ are both solutions to \cref{eq:eulerFlowII}, by \cref{thm:uniqFlowII} it holds $\phi^1=\phi^2$, concluding the proof.
\end{proof}

\subsection{Continuity of the solution map}
We conclude our work by proving that solutions to $\cref{eq:rougheuler}$ are continuous with respect to initial parameters $w_0$, $\sigma$ and $\mathbf{Z}$. This is not only important per se, but also immediately yields a Wong-Zakai type result for our problem.

\begin{theorem}\label{thm:WZ}
    Let $(L^\infty_{t,x})^{\text{weak},*}$ be  $L^\infty_{t,x}$ equipped with the weak-$*$ topology. Then,
	the solution map 
	\begin{align*}
	L^\infty \times (C^{3})^M \times \rpspace 
	& \to 
	(L^\infty_{t,x})^{\text{weak},*} \cap C_t W^{-1,1} \\
	(w_0,\sigma,\mathbf{Z}) 
	& \mapsto 
	w
	\end{align*}
	is continuous. 
	
	Moreover, if $\{B^{H,N}\}_{N\in\N}$ is a piecewise linear approximation of a fractional Brownian motion $B^H$, with Hurst parameter $H\in(\frac{1}{3},\frac{1}{2}]$, then the solutions $\{w^N\}$ of \cref{eq:origEuler} with $\dot{Z} = \dot{B}^{H,N}$ converge almost surely to the solution $w$ of 
	\begin{equation}\label{eq:wongzakai}
	\begin{cases}
	d w + u \sg w dt = \sum_{j=1}^M \sigma_j \sg w \circ\! d B^{H,j} , \\
	w(0) = w_0 , \quad u = K_{BS}*w .
	\end{cases}
	\end{equation}
\end{theorem}

\begin{proof}
We start by showing that, given a converging sequence of parameters, the solutions depending on it converge themselves. 
To this aim, consider $\{w_0^N, \sigma^N, \mathbf{Z}^N\} \subseteq L^\infty \times (C^3)^M \times \rpspace$ and assume it converges to some $\{w_0, \sigma, \mathbf{Z}\}$ in the same space. 
By well-posedness of \cref{eq:rougheuler}, we know that for each $N\in\N$ there exists a solution $w^N$ associated to the $N$-th element of the sequence. 
In particular, since the bounds we found only depend on initial parameters, the sequence of solutions $\{w^N\}$ is uniformly bounded in $L^\infty_{t,x} \cap C^{p\var}_T W^{-1,1}$: this ensures us that there exists a subsequence $\{w^{N_k}\}$ converging strongly to $w\in C_t W^{-1,1}$ (cf. \cref{lemma:compactnessArg}). 

We conclude by noting that the unbounded rough drivers $(A^{N,1}, A^{N,2})$, associated to $(\sigma^N,\mathbf{Z^N})$, converge, making us able to prove that $w$ solves \cref{eq:rougheuler} like we did for the existence proof, and that, by uniqueness of solutions, it is the whole sequence $\{w^N\}$ that converges.

At this point, the second statement if trivial: we know (see e.g. \cite{FH20}) that piecewise linear are one of those approximations of a fractional Brownian motion such that their canonical lift $(B^{H,N}, \int_0^\cdot B^{H,N} \otimes d B^{H,N})$ converges to $(B^H, \mathbb{B}^{H})$ in the rough path topology almost surely. 
\end{proof}

\begin{remark}
    In the case $H=\frac{1}{2}$, $B^H$ is a standard Brownian motion. 
    Then it is well known that piecewise linear approximations and their canonical lifts converge to Stratonovich enhanced Brownian motion, $(B,\mathbb{B}^{\textup{Strat}})$. 
    Since (Stratonovich) integration against an enhanced Brownian motion coincides almost surely with standard Stratonovich integration, this yields that $w^N$ associated to $\mathbf{B}^{H,N}$ converge to a probabilistically strong solution of \cref{eq:wongzakai}. 
    In this case, we end up recovering the result obtained in \cite{BFM16}.
\end{remark}

\begin{remark}
    The reasoning behind the proof of the Wong-Zakai statement in \cref{thm:WZ} actually holds in more general cases. For example, it holds for Gaussian rough paths satisfying the hypothesis of \cite[Theorem 10.4]{FH20}, for which it can be proven that piecewise linear approximations converge in rough path topology.
\end{remark}


\newpage
\appendix


\section{A compactness argument}

We prove here a compactness result, which plays a crucial role in \cref{th:vanishvisc}. 

\begin{lemma}\label{lemma:compactnessArg}
    Consider a sequence $\{g_n\}_{n\in \N} \subset L^{\infty}([0,T]\times \T^2) \cap C^{p\var}(W^{-1,1})$ such that $\norm{g_n}{L^{\infty}_{t,x}}$ is uniformly bounded and there exist $L>0$, two controls $\co$ and $\coL$ satisfying 
    \begin{equation*}\label{AssumptionCompact}
    \norm{(g_n)\st}{W^{-1,1}}\le \co(s,t)^{1/p} 
    \qquad
    \forall n\in\N, (s,t)\in\Delta_T \text{ s.t. }\coL(s,t)\le L .
    \end{equation*}
    Then, $\{g_n\}_n$ is compact in $C([0,T],W^{-1,1})$. 
\end{lemma}

\begin{proof}
    For all $\alpha \in (0,L]$, define the operator $J_{\alpha}f(s)=\frac{1}{\alpha}\int_s^{s+\alpha}f(r)dr$, where we impose $f(r)=f(T)$ for all $r>T$. The following hold true: 
    \begin{enumerate}
        \item For any fixed $\alpha$, $J_{\alpha}g_n: [0,T] \to W^{-1,1}$ are equicontinuous functions;
        \item As $\alpha$ vanishes, $\sup_{n\in \N}\norm{J_{\alpha}g_n-g_n}{C_tW^{-1,1}_x}\to 0$;
        \item For all $s \in [0,T]$ and $\alpha \in (0,L]$, $\left\{J_{\alpha}g_n(s)\right\}_n$ is relatively compact in $W^{-1,1}$;
        \item There exists a subsequence $g_{n_k}$ that converges weakly-* to some $g\in L^{\infty}([0,T]\times \T^2)$.
    \end{enumerate}
    The first and second properties follow from the hypothesis, the third one from $W^{1,1}$ being compactly embedded into $L^1$, and the last one from standard results in functional analysis. 

    Let now $\alpha$ be fixed: Ascoli-Arzelà theorem yields a subsequence, denoted again by $\{J_\alpha g_{n_k}\}$, such that $J_{\alpha}g_{n_k}$ converges to some $h$ in $C_tW^{-1,1}_x$. Then, point $4$ above and the continuity of $J_\alpha$ imply that $h=J_{\alpha}g$. In particular, any subsequence of $J_{\alpha}g_{n_k}$ has a sub-subsequence that converges to $h$; it follows that the full sequence $J_{\alpha}g_{n_k}$ must converge to $h$.
    
    Eventually, the result follows from the inequality
    \begin{multline*}
        \norm{g_{n_k}-g_{n_l}}{C_t W^{-1,1}}
        \le 
        \norm{g_{n_k}-J_{\alpha}g_{n_k}}{C_tW^{-1,1}} \\
        + \norm{J_{\alpha}g_{n_k}-J_{\alpha}g_{n_l}}{C_tW^{-1,1}_x} + \norm{J_{\alpha}g_{n_l}-g_{n_l}}{C_tW^{-1,1}_x} .
    \end{multline*}
\end{proof}

\section{Sewing Lemma and Rough Gronwall Lemma}

Sewing Lemma is crucial in rough paths theory since it allows for the building of rough integrals. We report a proper version of the Sewing Lemma, which can be found in \cite{DEYA20193577}. Its proof is a straightforward modification of \cite[Lemma 4.2]{FH20}.

\begin{lemma}\label{SewingLemma}
    Let $\xi \in (0,1)$, $L>0$ and $V$ be a Banach space. Consider two controls  $\co,\coL$ and a $2$-index function $h:\Delta_T\to V$ such that 
    \begin{equation*}
        |\delta h_{s,u,t}|\coloneqq |h_{s,t}-h_{s,u}-h_{u,t}|\le \co(s,t)^{\frac{1}{\xi}},
    \end{equation*}
    for all $(s,u,t)\in \triple$ with $\coL(s,t)\le L$.\\
    Then, there exists a unique path $I(h):[0,T]\to V$ with $I(h)_0=0$ such that the 2-index function
    $\Lambda_{s,t}\coloneqq I(h)_{s,t}-h_{s,t} \in C^{\xi\var,2}_{\coL,L}$.
    Moreover, there exists $C_{\xi}>0$ such that 
    \begin{equation*}
        |\Lambda\st|\le C_{\xi}\co(s,t)^{\frac{1}{\xi}},
    \end{equation*}
    for all $(s,t)\in \Delta_T$ with $\coL(s,t)\le L$.\\
    In addition if $h \in C^{p\var,2}_{\coL,L}(V)$ for some $p\geq\xi$, then $I\in C^{p\var}(V)$.
\end{lemma}

\begin{corollary} 
    Consider $h:\Delta_T\to V$ such that $h \in C^{\xi\var,2}_{\coL,L}(V)$ and it satisfies the assumptions of \cref{SewingLemma}. Then, 
    \begin{equation*}\label{ineq: corSewing}
        |h\st|\le C_{\xi}\co(s,t)^{\frac{1}{\xi}},
    \end{equation*}
    for all $(s,t)\in \Delta_T$ s.t. $\coL(s,t)\le L$.
\end{corollary}
The second reported result is a generalization of the Rough Gronwall Lemma firstly proved in \cite{DEYA20193577}.
\begin{lemma}\label{Rough Gronwall}
Consider a function $G:[0,T] \to \R^+$, three controls $\co_1,\co_2,\co_3$ and some constants $L,C'>0$, $C\geq 1$, $k>k'\geq 1$.\\
Suppose that the following conditions hold:
\begin{enumerate}
    \item $\co_2(s,t)\le \co_1(s,t)$ for all $(s,t) \in \Delta_T$;
    \item For all $(s,t)\in \Delta_T$ with $\co_1(s,t)\le L$,
    \begin{equation*}
       G_{s,t}\le C\big(\sup_{r \in [0,T]}G_r+C'\big) \co_1(s,t)^{\frac{1}{k}}+\co_2(s,t)^{\frac{1}{k'}}+\co_3(s,t). 
    \end{equation*}
\end{enumerate}
Then, it holds
\begin{equation*}
\begin{aligned}    
    \sup_{t\in [0,T]} G_t \le 2 \exp\Big(\frac{\co_1(0,T)}{\alpha L}\Big)\Big\{G_0 &+ \sup_{t \in [0,T]}\Big(\co_3 (0,t)\exp\Big(\frac{-\co_1(0,t)}{\alpha L}\Big)\Big)\\
    &+ 
    \sup_{t \in [0,T]}\Big(\co_2 (0,t)^{\frac{1-\theta}{k'}}\exp\Big(\frac{-\co_1(0,t)}{\alpha L}\Big)\Big)+C'\Big\},
    \end{aligned}
\end{equation*}
where $\theta=\frac{k'}{k}$ and $\alpha=min\Big(1,\frac{1}{L(2Ce^2)^{k}}\Big)$.
\end{lemma}
\begin{proof}
    For any $t \in [0,T]$, consider the following functions:
    \begin{align*}
        G_{\le t} & \coloneqq \sup_{s\in[0,t]}G_s , \\ 
        H_t & \coloneqq \exp\Big (-\frac{\co_1(0,t)}{\alpha L}\Big) \Big(G_{\le t}+\co_2(0,t)^{\frac{1-\theta}{k'}}+C'\Big) , \\
        H_{\le t} & \coloneqq \sup_{s\in[0,t]}H_s.
    \end{align*}
    Since $\co_1$ is a continuous function, then there exists a partition $P=\{t_0\coloneqq 0,\dots,t_N\coloneqq T\}$ such that $\co_1(t_{j-1},t_j)=\alpha L$ for all $j\in \{1,\cdots,N-1\}$.\\
    Consider $t\in [t_{i-1},t_i]$, then 
    \begin{align*}
        G_{0,t} 
        ={} & 
        \sum_{j=0}^{i-2}G_{t_j,t_{j+1}}+G_{t_{i-1},t} \\
        \le{} & 
        \sum_{j=0}^{i-2}\Big[
        C(G_{\le t_{j+1}}+C') \co_1(t_j,t_{j+1})^{\frac{1}{k}}
        +\co_2(t_j,t_{j+1})^{\frac{1}{k'}}
        +\co_3(t_j,t_{j+1})\Big] \\
        & +C(G_{\le t}+C') \co_1(t_{i-1},t)^{\frac{1}{k}}+\co_2(t_{i-1},t)^{\frac{1}{k'}}+\co_3(t_{i-1},t) \\
        \le{} &
        \sum_{j=0}^{i-2}\Big[C (\alpha L)^{\frac{1}{k}}(G_{\le t_{j+1}}+C') +\co_2(0,t_{j+1})^{\frac{1-\theta}{k'}}\co_1(t_j,t_{j+1})^{\frac{1}{k}}\Big] \\
        & +C(G_{\le t_i}+C')(\alpha L)^{\frac{1}{k}}
        + \co_2(0,t_i)^{\frac{1-\theta}{k'}}\co_1(t_{i-1},t)^{\frac{1}{k}}+\co_3(0,t) \\
        \le{} &
        C(\alpha L)^{\frac{1}{k}} \sum_{j=0}^{i-1}\Big[G_{\le t_{j+1}}+\co_2(0,t_{j+1})^{\frac{1-\theta}{k'}}+C'\Big]+\co_3(0,t),
    \end{align*}
    where the first and second inequalities are due to the hypothesis.\\
    Now,
    \begin{align*}
        \sum_{j=0}^{i-1} \Big[ G_{\le t_{j+1}} +\co_2(0,t_{j+1})^{\frac{1-\theta}{k'}} + C' \Big] 
        ={} & 
        \sum_{j=0}^{i-1} H_{t_{j+1}} \exp\Big(\frac{\co_1(0,t_{j+1})}{\alpha L}\Big) \\\le{} &
        H_{\le T}\sum_{j=0}^{i-1}\exp(j+1) \\
        \le{} & 
        H_{\le T}\exp(i+1).
    \end{align*}
    Therefore, 
    \begin{equation}\label{Eq1RGL}
        G_t\le G_0+\co_3(0,t)+C(\alpha L)^{\frac{1}{k}}H_{\le T}\exp(i+1).
    \end{equation}
    We exploit \cref{Eq1RGL} to estimate $H_{\le T}$. Indeed, given $t\in [t_{i-1},t_i]$, then
    \begin{align*}
        H_t 
        ={} & 
        \exp\Big (-\frac{\co_1(0,t)}{\alpha L}\Big) \Big(G_{\le t}+\co_2(0,t)^{\frac{1-\theta}{k'}}+C'\Big) \\
        \le{} & 
        \exp\Big(-\frac{\co_1(0,t)}{\alpha L}\Big) \\
        & \Big[G_0+\co_3(0,t) + C(\alpha L)^{1/k}H_{\le T}\exp(i+1) +\co_2(0,t)^{\frac{1-\theta}{k'}}+C'\Big] \\
        \le{} & 
        G_0+\sup_{r \in [0,t]}\Big(\co_3(0,r)\exp\Big (-\frac{\co_1(0,r)}{\alpha L}\Big)\Big) + \\
        & \sup_{r \in [0,t]}\Big(\co_2(0,r)^{\frac{1-\theta}{k'}}\exp\Big (-\frac{\co_1(0,r)}{\alpha L}\Big)\Big) +C(\alpha L)^{\frac{1}{k}}\e^2H_{\le T}+C'.
    \end{align*}
    Therefore, taking the supremum in the time variable $t$ implies that
    \begin{multline}\label{Eq2RGL}
        H_{\le T}
        \le 
        2G_0 + 2 \sup_{r \in [0,T]} \Big( \co_3(0,r) \exp\Big(-\frac{\co_1(0,r)}{\alpha L}\Big) \Big) \\
        +2\sup_{r \in [0,T]}\Big(\co_2(0,r)^{\frac{1-\theta}{k'}}\exp\Big (-\frac{\co_1(0,r)}{\alpha L}\Big)\Big)+2C'.
    \end{multline}
    At last, notice that 
    \begin{equation*}
    \begin{aligned}
        G_{\le T}\le H_T \exp\Big(\frac{\co_1(0,T)}{\alpha L}\Big)\le H_{\le T} \exp\Big(\frac{\co_1(0,T)}{\alpha L}\Big),
    \end{aligned}
    \end{equation*}
     which proves the result once it is combined with \cref{Eq2RGL}. 
\end{proof}

\newpage
\printbibliography

\end{document}